\documentclass[reqno,11pt]{amsart}

\usepackage[top=2.0cm,bottom=2.0cm,left=3cm,right=3cm]{geometry}
\usepackage{amsthm,amsmath,amssymb,dsfont}
\usepackage{mathrsfs,amsfonts,functan,extarrows,mathtools}
\usepackage[colorlinks]{hyperref}
\usepackage{stmaryrd}
\SetSymbolFont{stmry}{bold}{U}{stmry}{m}{n}

\usepackage{marginnote}
\usepackage{xcolor}
\usepackage{esint}
\usepackage{graphicx}
\usepackage{bm}
\usepackage{soul}
\newtheorem{theorem}{Theorem}[section]
\newtheorem{proposition}{Proposition}[section]

\newtheorem{remark}{Remark}[section]
\newtheorem{lemma}{Lemma}[section]

\numberwithin{equation}{section}
\allowdisplaybreaks

\def\L{\mathcal{L}}
\def\M{\sqrt{M}}

\def\R{\mathbb{R}}

\def\P{\mathbb{P}}

\def\v{\varepsilon}

\def\l{\langle}
\def\r{\rangle}

\def\S{\mathbb{S}}

\def\up{\textup}
\def\n{\nabla}
\def\p{\partial}

\def\L{\mathcal{L}}

\def\a{\alpha}
\def\b{\beta}

\def\c{\cdot}

\def\E{\mathcal{E}}
\def\M{\sqrt{M}}

\newcommand{\bu}{\mathbf{u}}
\newcommand{\btheta}{\bm{\theta}}
\newcommand{\brho}{\bm{\rho}}

\newcommand{\bff}{\mathbf{f}}
\newcommand{\bg}{\mathbf{g}}
\newcommand{\bh}{\mathbf{h}}

\newcounter{wronumber}

\begin{document}

\author[Zhendong Fang]{Zhendong Fang}
\address[Zhendong Fang]
        {\newline School of Mathematics, South China University of Technology, Guangzhou, 510641, P. R. China}
\email{zdfang@scut.edu.cn}

\author[Kunlun Qi]{Kunlun Qi}
\address[Kunlun Qi]
        {\newline School of Mathematics, University of Minnesota - Twin Cities, Minneapolis, MN, 55455, USA}
\email{kqi@umn.edu}

\title[Hydrodynamic Limit of the Boltzmann equation for gas mixture] {From the Boltzmann equation for gas mixture to the two-fluid incompressible hydrodynamic system}

\keywords{Hydrodynamic limit; Boltzmann equation; Gas mixture; Incompressible Navier-Stokes-Fourier equation; Dimensionless analysis}

\subjclass[2020]{Primary 35B25; 35Q30; 35Q20.}


\begin{abstract}
In this paper, we study the hydrodynamic limit transition from the Boltzmann equation for gas mixtures to the two-fluid macroscopic system.
Employing a meticulous dimensionless analysis, we derive several novel hydrodynamic models via the moments' method.
For a certain class of scaled Boltzmann equations governing gas mixtures of two species, we rigorously establish the two-fluid incompressible Navier-Stokes-Fourier system as the hydrodynamic limit. This validation is achieved through the Hilbert expansion around the global Maxwellian and refined energy estimates based on the Macro-Micro decomposition.
\end{abstract}

\maketitle


\tableofcontents


\section{Introduction}
\subsection{The Boltzmann equation for gas mixture}

The Boltzmann equation for the gas mixture (BEGM) describes the time evolution of the density distribution of the gas mixture of two different species \cite{Cercignani87}, which reads:
\begin{equation}\label{BEGM}
\begin{cases}
\p_t f_1+v\c\n_x f_1 = Q_{11}(f_1,f_1) + Q_{12}(f_1,f_2),\\
\p_t f_2+v\c\n_x f_2 = Q_{22}(f_2,f_2) + Q_{21}(f_2,f_1),\\
\end{cases}
\end{equation}
where $f_l(t,x,v),\,l \in \{1,2\}$ denotes the density distribution functions of the gas molecules at time $t\geq0$, with position $x\in \mathbb{R}^3$ and velocity $v\in\R^3$. The collision operator $Q_{ln}(f_l,g_n)$ is given by, for $l,n \in \{1,2\}$,
\begin{equation}\label{Qln}
Q_{ln}(f_l,g_n) = \frac{1}{2} \int_{\R^3}\int_{\S^2}  B_{ln}(v-v_*,\sigma) \left( f'_l g'_{n*} + f'_{l*} g'_n - f_l g_{n*} - f_{l*} g_n \right) \,d\sigma \,dv_*,
\end{equation}
where $f_l=f_l(t,x,v),\,g_{n*}=g_n(t,x,v_*),\,f'_l=f_l(t,x,v'),$ $g'_{n*}=g_n(t,x,v'_*)$ with $(v', v'_*)$ and $(v, v_*)$ representing the velocity pairs before and after the collisions, respectively. For simplicity, we assume that particles of different species have the same mass $m$, but different radii (see \cite[Chapter II, Section 4]{Cercignani87}) such that the conservation of momentum and energy hold in the sense that
\begin{equation*}
v' + v_{*}' = v + v_{*}, \quad  |v'|^{2} + |v_{*}'|^{2} = |v|^{2} + |v_{*}|^{2}.
\end{equation*}
This allows us to express $(v', v_*')$ in terms of $(v, v_*)$ and unit vector $\sigma \in \mathbb{S}^{2}$ using the following relations:
\begin{equation*}
\begin{aligned}
v'= & v - [(v-v_*)\c\sigma] \sigma,\\
v'_*= & v_* + [(v-v_*)\c\sigma] \sigma.
\end{aligned}
\end{equation*}
The collision kernel $B_{ln}(v-v_*,\sigma)$ describes the intensity of collisions, which will be assumed to satisfy the following properties: for $l,n \in \{1,2\}$,
\begin{itemize}
    \item The collisions are symmetric between species:
    \begin{equation*}
        B_{ln}(v-v_*,\sigma) = B_{nl}(v-v_*,\sigma).
    \end{equation*}
    \item $B_{ln}$ can be separated into the kinetic part $\Phi$ and angular part $b$ in the case of the inverse power law:
    \begin{equation*}
		B_{ln}(v-v_*,\sigma) = b_{ln}(\cos\theta)\,\Phi_{ln}(|v-v_*|), \quad \text{with} \ \cos\theta=\sigma \cdot \frac{v-v_*}{|v-v_*|},
    \end{equation*}
    where kinetic collision part $\Phi(|v-v_*|)=|v-v_*|^{\gamma}$ includes hard potential $ (\gamma>0) $, Maxwellian molecule $ (\gamma =0) $ and soft potential $ (\gamma<0) $. Note that, for the rigorous justification throughout this paper, we consider the well-known hard sphere model, i.e., $\gamma = 1$, and $b_{ln}$ to satisfy the Grad's cutoff assumption.

\end{itemize}
We refer the readers for more details about the collision kernel $B_{ln}$ in \cite{Villani02}.


\subsection{Dimensionless form of the BEGM equation}
\label{subsec:dimensionless}

To figure out the scaling that will be applied in the hydrodynamic limit process, we present the complete derivation of the dimensionless form of the BEGM equation.

Suppose the reference temperature $T_0$ and reference average macroscopic density $R_0$ are identical with the mixture particles, we can first define the macroscopic velocity $U_0$ and microscopic velocity $c_0$ as follows: for given macroscopic time $t_0$, length $L_0$, and Boltzmann contact $k$,
\begin{itemize}
    \item The macroscopic velocity $U_0:=\frac{L_0}{t_0}$;
    \item The microscopic velocity $c_0:=\sqrt{\frac{5kT_0}{3m}}$.
\end{itemize}
Then, we can introduce the dimensionless variables of time $\tilde{t}$, space $\tilde{x}$ and velocity $\tilde{v}$:
\begin{equation*}
\tilde{t}=\frac{t}{t_0}, \quad \tilde{x}=\frac{x}{L_0},\quad \tilde{v}=\frac{v}{c_0},
\end{equation*}
and the dimensionless density distribution function $\tilde{f}_l(\tilde{t},\tilde{x},\tilde{v})$ follows that
\begin{equation*}
\tilde{f}_l(\tilde{t},\tilde{x},\tilde{v})=\frac{R_0}{c_0^3}f_l(t,x,v), \quad \text{for} \,\, l \in \{1,2\}.
\end{equation*}

To complete the derivation of the dimensionless BEGM equation, the following parameters will be introduced
\begin{itemize}
    \item Mean free time $\tau_{l}$. Since the Boltzmann kernel has units of the reciprocal product of density, then we define a mean free time $\tau_{l}$ of one single type of particle
    \begin{equation*}
    \int_{\R^3}\int_{\R^3}\int_{\S^2}\mathcal{M}_{[R_0,0,T_0]}(v)\mathcal{M}_{[R_0,0,T_0]}(v_*)B_{ll}(|v-v_*|,\omega) \,d\omega \,dv_* \,dv=\frac{R_0}{\tau_{l}},
    \end{equation*}
    where $\mathcal{M}_{[R_0,0,T_0]}(v)$ the Maxwellian distribution:
    \begin{equation*}
    \mathcal{M}_{[R_0,0,T_0]}(v):=\frac{R_0}{(2\pi T_0)^{\frac{3}{2}}} \up{e}^{-\frac{|v|^2}{2T_0}}.
    \end{equation*}

    \item Strength of interactions between species $\delta(l,n)$. Note that the particles of different species will also collide with each other in the gas mixture, we need to introduce the following dimensionless parameter $\delta(l,n)$ to describe the strength of interactions between the particles of different species \cite{Saint-Raymond19}: for $l,n \in \{1,2\}$,
    \begin{equation}\label{Definition of delta}
    \delta(l,n)
    =\left\{
    \begin{array}{cc}
    1, & \qquad \up{if}\,\, l=n,\\[4pt]
    \Bar{\delta},& \qquad \up{if}\,\, l\neq n,\\[4pt]
    \end{array}
    \right.
    \end{equation}
    with the constant $\Bar{\delta} > 0$.
    Hence, we can define the mean free time for the gas mixture as
    \begin{equation*}
    \tau_{ln} = \frac{\sqrt{\tau_l\tau_n}}{\delta^2(l,n)}.
    \end{equation*}

    \item Mean free path $\lambda_{ln}$. The scale of the average time that particles in the equilibrium density $\mathcal{M}_{[R_0, 0, T_0]}$ spend traveling freely between
    two collisions, which is related to the length scale of the mean free path $\lambda_{ln}$
    \begin{equation*}
    \lambda_{ln}=c_0\tau_{ln}.
    \end{equation*}
    From the definition of $\tau_{ln}$, we have $\lambda_{ln}=\lambda_{nl}$ for each $l,n \in \{1,2\}$ and denote $\lambda_l:=\lambda_{ll}$ for convenience.
\end{itemize}

Finally, the corresponding dimensionless collision operator is derived as follows:
\begin{equation*}
\tilde{Q}_{ln}(\tilde{f}_l,\tilde{g}_n) = \frac{1}{2}\int_{\R^3}\int_{\S^2}
\tilde{B}_{ln}(|\tilde{v}-\tilde{v}_*|,\omega) (\tilde{f}'_l \tilde{g}'_{n*} + \tilde{f}'_{l*} \tilde{g}'_n - \tilde{f}_l\tilde{g}_{n*} - \tilde{f}_{l*}\tilde{g}_n)  \,d\omega \,d\tilde{v}_* \,,
\end{equation*}
where the dimensionless Boltzmann collision kernel $\tilde{B}_{ln}$ is
\begin{equation*}
\tilde{B}_{ln}(\tilde{v}-\tilde{v}_*, \sigma) = R_0 \tau_{ln} B_{ln}(v-v_*,\sigma).
\end{equation*}

Dropping all tildes, we deduce the BEGM equation in the dimensionless form:
\begin{equation}\label{Scaled BE equaion-0}
\begin{cases}
\up{St}\, \p_t f_1+v\c\n_xf_1=\frac{1}{\up{Kn}_1}Q_{11}(f_1,f_1)+\frac{\Bar{\delta}^2}{\sqrt{\up{Kn}_1 \up{Kn}_2}}Q_{12}(f_1,f_2),\\[6pt]
\up{St}\, \p_t f_2+v\c\n_xf_2=\frac{1}{\up{Kn}_2}Q_{22}(f_2,f_2)+\frac{\Bar{\delta}^2}{\sqrt{\up{Kn}_1 \up{Kn}_2}}Q_{21}(f_2,f_1),\\[4pt]
\end{cases}
\end{equation}
where the Strouhal number $\up{St}$ and Knudsen number $\up{Kn}_l$ \cite{SRL09} are give by
\begin{equation*}
\up{St} = \frac{L_0}{c_0 t_0}, \quad \up{Kn}_l = \frac{\lambda_l}{L_0}.
\end{equation*}

Based on \cite{Saint-Raymond19}, we will distinguish the following three cases:
\begin{itemize}
    \item Strong inter-species interactions: $\Bar{\delta}=O(1)$;\\[4pt]
    \item Weak inter-species interactions: $\Bar{\delta}=o(1)$ and $\frac{\Bar{\delta}}{(\up{Kn}_1 \up{Kn}_2\up{St}^2)^{\frac{1}{4}}}$ is unbounded;\\[4pt]
    \item Very weak inter-species interactions: $\Bar{\delta}=O\big( (\up{Kn}_1\up{Kn}_2 \up{St}^2)^{\frac{1}{4}} \big)$.\\
\end{itemize}

In this paper, we only consider the third case with the same collision kernel, i.e.,
\begin{equation*}
    B(|v-v_*|,\sigma):=B_{ln}(|v-v_*|,\sigma), \quad Q(f_l,g_n):=Q_{ln}(f_l,g_n),
\end{equation*}
for $l,n \in \{1,2\}$.

Furthermore, for $l \in \{1,2\}$, we choose $\up{St}=\v$, $\up{Kn}_l=\v^{c_l}$, and $\Bar{\delta} = \v^q$ with $c_l,\,q>0$ and consider
\begin{equation*}
\frac{\Bar{\delta}}{(\up{Kn}_1 \up{Kn}_2 \up{St}^2)^{\frac{1}{4}}}=1,
\end{equation*}
which implies that
\begin{equation*}
4q=2+c_1+c_2.
\end{equation*}

Therefore, the scaled BEGM equation \eqref{Scaled BE equaion-0} can be rewritten as
\begin{equation}\label{Scaled BE equaion-1}
\begin{cases}
\v\p_t f^\v_1+v\c\n_xf^\v_1=\v^{-c_1}Q(f^\v_1,f^\v_1)+\v Q(f^\v_1,f^\v_2),\\[4pt]
\v\p_t f^\v_2+v\c\n_xf^\v_2=\v^{-c_2}Q(f^\v_2,f^\v_2)+\v Q(f^\v_2,f^\v_1).
\end{cases}
\end{equation}
or, that is to say,
\begin{equation}\label{Formally-0}
\begin{cases}
Q(f^\v_1,f^\v_1) = \v^{1+c_1}\p_t f^\v_1 + \v^{c_1}v\c\n_xf^\v_1 - \v^{1+c_1}Q(f^\v_1,f^\v_2),\\[4pt]
Q(f^\v_2,f^\v_2)=\v^{1+c_2}\p_t f^\v_2 + \v^{c_2}v\c\n_xf^\v_2 -\v^{1+c_2}Q(f^\v_2,f^\v_1).
\end{cases}
\end{equation}
Then, suppose $f^\v_l\to f_l$ as $\v\to 0$, the right-hand side of \eqref{Formally-0} vanishes, which formally implies that
\begin{equation}\label{Formally-1}
Q(f^\v_l,f^\v_l)= \v^{1+c_l}\p_t f^\v_l+\v^{c_l} v\c\n_xf^\v_l -\v^{1+c_2}Q(f^\v_l,f^\v_n)\to 0,
\end{equation}
for $l,n \in \{1,2\}$ and $l \neq n$; on the other hand, we have
\begin{equation}\label{Formally-2}
Q(f^\v_l,f^\v_l)\to Q(f_l,f_l), \quad \up{as} \quad  \v \to 0.
\end{equation}
Thus, combining with \eqref{Formally-1} and \eqref{Formally-2}, we obtain
\begin{equation}\label{Formally-3}
Q(f_l,f_l)=0.
\end{equation}

\begin{remark}
Note that, according to the H-theorem of the Boltzmann equation \cite[Theorem 3.1]{GFSRL05}, the following conditions are equivalent:
\begin{itemize}
    \item $Q(f_l,f_l)=0$ a.e.;\\[-5pt]
    \item $\int_{\R^3}Q(f_l,f_l) \ln f_l \,dv = 0$;\\[-5pt]
    \item $f_l$ is the Maxwellian distribution function, i.e.
    \begin{equation*}
     f_l = \mathcal{M}_{[\rho_l, u_l, \theta_l]}(v):= \frac{\rho_l}{(2\pi\theta_l)^{\frac{3}{2}}}\up{e}^{-\frac{|v-u_l|^2}{2\theta_l}},
    \end{equation*}
    for some $\rho_l,\,\theta_l>0$ and $u_l\in\R^3$.
\end{itemize}
\end{remark}

\subsection{Previous results and our contributions}
\label{subsec:previous}

In this subsection, we undertake a comprehensive review of prior research pertaining to the well-posedness and hydrodynamic limit of the classical Boltzmann equation, alongside relevant results regarding its application to gas mixtures, which serve as a significant impetus for our work. Furthermore, our contribution and novelty of this paper will be introduced as well.\\[-5pt]

\textit{Previous results of ``well-posedness"}
Standing as a cornerstone in kinetic theory, the Boltzmann equation has been attracting enduring attention for a long history, particularly the studies regarding its well-posedness. In the realm of weak solutions, in \cite{DRLPL89}, DiPerna-Lions laid the foundational groundwork by establishing renormalized solutions for the Boltzmann equation, accommodating general initial data under Grad's cutoff assumption. Furthermore, Alexandre-Villani extended this understanding to the case that encompasses long-range interaction kernels \cite{ARVC02}. In terms of the initial boundary problem, Mischler proved the well-posedness of the Boltzmann equation with Maxwell reflection boundary condition for the cutoff case in \cite{MS10}.
Within the realm of classical solutions, Ukai in \cite{US74} achieved a breakthrough with the first attainment of global well-posedness in the close-to-equilibrium sense, specifically for collision kernels featuring cut-off hard potentials. By using the nonlinear energy method, the same type of result for the soft potential case for both a periodic domain and for whole space was proved by Guo \cite{GY03, GY04}. In the absence of the Grad's cutoff assumption, the existence and regularity of global classical solution near the equilibrium for the whole space were obtained by Gressman-Strain \cite{GPTSRM11} and Alexandre-Morimoto-Ukai-Xu-Yang \cite{ARMYUSXCJYT12}.
Recent years have seen significant strides in understanding strong and mild solutions of Boltzmann equations within bounded domains under various boundary conditions \cite{DRJLWXLLQ22, DRJLSQSSSRM21, GY10, KCLD18}. For further exploration, we refer readers to additional pertinent progress concerning the regularity and other types of Boltzmann models \cite{ICSL20, ICSLE22, Qi21_soft, Qi22}. \\[-5pt]

\textit{Previous results of ``hydrodynamic limit"}
The hydrodynamic limit is the limiting process that connects the kinetic equation (scaled Boltzmann equation) with the fluid equation (Euler or Navier-Stokes equation). This concept can be traced back to Maxwell and Boltzmann, who initially founded the kinetic theory. The project of studying the hydrodynamic limit was then specifically formulated and addressed by Hilbert \cite{Hilbert1912}. It aims to derive the fluid models as particles undergo an increasing number of collisions, causing the Knudsen number approaches to vanish.

Based on the existence of renormalized solutions \cite{ARVC02, DRLPL89}, one type of framework for studying the hydrodynamic limit pertains to weak solutions, particularly proving that the renormalized solution of the Boltzmann equation converges to the weak solution of the Euler or Navier-Stokes equations. In \cite{BCGFLD91}, Bardos-Golse-Levermore started with the formal derivation of the fluid equations, including compressible Euler equations, and incompressible Euler and Navier-Stokes equations. They also initialed the so-called \textit{BGL program} to justify Leray's solutions of the incompressible Navier-Stokes equations from renormalized solutions \cite{BCGFLCD93}, which was somewhat completed by Golse-Saint Raymond under the cutoff assumption \cite{GFSRL04, GFSRL09}. A similar result was obtained by Arsenio \cite{AD12} for the non-cutoff case as well. More results following this methodology can be found in \cite{JNLCDMN10, JNMN17, HLBJJC23, MNSRL03}.

Another aspect of studying the hydrodynamic limit is from the perspective of classical solutions. The classical compressible Euler and Navier-Stokes equations can be formally derived from the scaled Boltzmann equation through the Hilbert and Chapman-Enskog expansions. The rigorous justification behind the asymptotic convergence was initially established by Caflish for the compressible Euler equations \cite{CRE80} and by De Masi-Esposito-Lebowitz for the incompressible Navier-Stokes equations \cite{DMAERLJL89}. By taking advantage of the energy method, Guo-Jang-Jiang also attained significant progress in understanding the acoustic limit  \cite{GYJJJN09, GYJJJN10, JJJN09}.
Additionally, research on strong solutions near equilibrium is another avenue of exploration in the hydrodynamic limit. Nishida \cite{NT78} established local-in-time convergence to the compressible Euler equations, while Bardos-Ukai in \cite{BCUS91}, as well as Gallagher-Tristani in their recent work \cite{GT2020}, derived solutions for the incompressible Navier-Stokes equations. For comprehensive previous results, we refer to \cite{GY06, BM15, JNXLJ15, JNXCJZHJ18, GYHFMWY21, JJKC21} and the references cited therein.\\[-5pt]

\textit{Previous results for ``Boltzmann equation for gas mixture"} In recent years, the study of the Boltzmann equation for the gas mixture has achieved tremendous progress due to its more physical significance in describing the real world. When the different species of gas particles possess the same mass, the linearized Boltzmann operator can be completely decoupled such that the usual tools in studying the Boltzmann equation of a single species can still be applied, for instance,
in \cite{GuoYan03}, Guo first studied the well-posedness of the Vlasov-Maxwell-Boltzmann system of two types of gases around the Maxwellian, the decay rate of which was obtained by Wang in \cite{Wang13}. Aoki-Bardos-Takata studied the existence of the Knudsen layer of the Boltzmann equation for a gas mixture with the zero bulk velocity in \cite{ABT03}, which was further extended by Bardos-Yang in \cite{BY12} for the case of general drifting velocity. On the other hand, when the mass of the gas particles of different species are not identical, we refer to the work by Sotirov-Yu in \cite{SY10}, where they studied the Boltzmann equation for gas mixture in one space dimension via the Green's function, as well as the work by Briant-Daus in \cite{BD16} showing the existence of the solution to Boltzmann equation for the gas mixture around the bi-Mawellian.

In contrast with the Boltzmann equation of one single species, there are few results concerning the hydrodynamic limit of the Boltzmann equation of gas mixture. It is worth mentioning the recent work by Wu-Yang \cite{WuYang23},
where, for the scaling case $\up{St}=O(1)$ and $\up{Kn}_l=o(\v)$ in \eqref{Scaled BE equaion-0}, the two-fluid compressible Euler equations was rigorously justified from the Boltzmann equation of gas mixture via the Hilbert expansion method. For more hydrodynamic limits from the Boltzmann equation for gas mixture, especially the formal derivation, we refer the readers to \cite{Dogbe08, BD15} and the references therein.\\[-5pt]

\textit{Mathematical challenges and our contributions}
Motivated by the previous results above, this paper aims to study the hydrodynamic limit of the Boltzmann equation for gas mixture, mainly focusing on the following two aspects:\\
(i) In terms of formal derivation,  this paper presents several novel hydrodynamic models, which, to the best of our knowledge, are derived for the first time from the Boltzmann equation for gas mixtures (BEGM equation) \eqref{Scaled BE equaion-1}, employing various scalings and expansion forms. Notably, these include the Euler-Fourier equations coupled with the Navier-Stokes equations, among others (see Theorem \ref{Main-theorem-1} for additional models). To achieve this, the delicate dimensionless analysis in Section \ref{subsec:dimensionless} has been applied, particularly focusing on describing the degree of collision between different species of particles. Inspired by \cite{Saint-Raymond19}, we introduce a dimensionless parameter $\delta(l,n)$ in \eqref{Definition of delta} to accurately quantify the strength of interactions between two species. \\
(ii) In terms of the rigorous justification, we prove the two-fluid incompressible Navier-Stokes-Fourier system \eqref{The Boltzmann equation for gaseous mixtures NSF equ} as the hydrodynamic model stemming from the BEGM equation \eqref{Scaled gas mixture BE system-0} under specific scaling and expansion forms (refer to Theorem \ref{Limit of Fluid equations}).
Here, we apply the methodology following the framework of the \textit{BGL program} mentioned above as well as leverage the known hydrodynamic limit result of the single-species Boltzmann equation \cite{JNXCJZHJ18}. More precisely, by selecting the particular expansion form \eqref{Special form solution} of the solution to the BEGM equation \eqref{Scaled gas mixture BE system-0}, we can derive the reminder system \eqref{Scaled gas mixture BE system-0g}  through prescribed scaling, the solution to which is shown to converge to the function constructed by the solutions to the incompressible Navier-Stokes-Fourier system \eqref{The Boltzmann equation for gaseous mixtures NSF equ}, as $\v \to 0$. To this end, the key difficulty lies in establishing the uniform energy estimates of the solution to the reminder system \eqref{Scaled gas mixture BE system-0g}; to address this, we apply the well-known Macro-Micro decomposition \cite{LYY04} and take full advantage of the dissipative structure of the linearized Boltzmann operator to close the refined energy estimates, thereby facilitating the proof of the limiting process through compact arguments.

\section{Main Results}
\label{sec:main-results}

\subsection{Notations}
\label{subsec:notations}

(i) $A\lesssim B$ denotes $A\leq C B$ for some generic constants $C > 0$ and $A \thicksim B$ denotes that there exist two generic constants $C_{1}, C_{2}>0$ such that $C_{1} A\leq B\leq C_{2}A$.\\

(ii) $\alpha=(\alpha_1,\alpha_2,\alpha_3)$ is the multi-index in $\mathbb{N}^3$ with $| \alpha | = \alpha_1 + \alpha_2 + \alpha_3$. The $\alpha^{th}$ partial derivative denoted by
 $$ \partial_x^\alpha= \partial_{x_1}^{\alpha_1} \partial_{x_2}^{\alpha_2} \partial_{x_3}^{\alpha_3}.$$
$\alpha \leq \tilde{\alpha}$ means each component of $\alpha \in \mathbb{N}^3$ is not greater than that of $\tilde{\alpha}$. $\alpha < \tilde{\alpha}$ means $\alpha \leq \tilde{\alpha}$ and $|\alpha| < |\tilde{\alpha}|$.\\

(iii) $L^{p}$ denotes the usual Lebesgue space, namely,
\begin{equation*}
  L_{x}^{p}=L^{p}(dx), \quad L_{v}^{p}(\omega)=L^{p}(\omega dv)
\end{equation*}
endowed with the norms:
\begin{equation*}
\begin{split}
    \|f\|_{L^{p}_{x}} =& \left(\int_{\R^{3}} |f(x)|^{p} \,dx \right)^{\frac{1}{p}} < \infty, \quad p \in [ 1, \infty ),\\[5pt]
    \|f\|_{L^{p}_{v}(\omega)} =& \left(\int_{\R^{3}}|f|^{p}\omega \,dv \right)^{\frac{1}{p}} < \infty, \quad p \in [ 1, \infty ), \\[5pt]
    \|f\|_{L^{\infty}_{x}}= & \text{ess} \sup\limits_{x\in\R^{3}}|f(x)|<\infty,
\end{split}
\end{equation*}
where the weight function $\omega$ is either $1$ or the collisional frequency $\nu(v)$ given by
\begin{equation*}
\nu(v)=\int_{\mathbb{R}^{3}}|v-v_{*}|M(v_{*}) \,dv_{*} \sim 1+|v| .
\end{equation*}
$L^{p}_{x}L^{q}_{v}(\omega)$ denotes Lebesgue space for $p,q\in[1,+\infty]$ (if $p = q$, $L^p_{x,v}(w) = L^p_x L^p_v (w)$) endowed with the norms:
\begin{equation*}
  \begin{aligned}
    \| f \|_{L^p_x L^q_v (w)} =
    \left\{
      \begin{array}{l}
        \left( \int_{\R^3} \| f (x, \cdot) \|^p_{L^q_v (w)} \,d x \right)^\frac{1}{p},  \qquad\quad\quad p,\,q \in [ 1, \infty ) \,, \\[10pt]
        \textrm{ess}\sup\limits_{x \in \R^3} \| f (x,\cdot) \|_{L^q_v (w)},  \qquad\qquad\quad\ \  p = \infty, \, q \in [1, \infty) \,, \\[10pt]
        \left( \int_{\R^3}  \textrm{ess} \sup\limits_{v \in \R^3} |f(x,v)|^p \,d x \right)^\frac{1}{p},  \quad\quad p \in [ 1, \infty ), \, q = \infty \,, \\[10pt]
        \textrm{ess}\sup\limits_{(x,v) \in \R^3 \times \R^3} |f (x,v) w(v)|, \qquad\quad \ p = q = \infty \,.
      \end{array}
    \right.
  \end{aligned}
\end{equation*}
$H^s_x L^2_v$ and $H^s_x L^2_v (\nu)$ denote the Sobolev spaces  endowed with the norms:
\begin{equation*}
  \begin{aligned}
    \| f \|_{H^s_x L^2_v} = \left( \sum_{|\alpha| \leq s} \| \partial^\alpha_x f \|^2_{L^2_{x,v}} \right)^\frac{1}{2} ,\quad
    \| f \|_{H^s_x L^2_v (\nu)} = \left( \sum_{|\alpha| \leq s} \| \partial^\alpha_x f \|^2_{L^2_{x,v}(\nu)} \right)^\frac{1}{2}.
  \end{aligned}
\end{equation*}

(iv) $ \l\cdot , \cdot \r_{x,v} $, $\l \cdot , \cdot \r_{v}$ and $\l \cdot , \cdot \r_{x}$ denote the inner product in $L^2_{x,v}$, $L^2_v$ and $L^2_x$.


\subsection{Statement of main results}
\label{subsec:statement}

Our main results will be stated first of all.

Inspired by the Galilean transformation, we seek a special form of the solution to \eqref{Scaled BE equaion-1} around the global Maxwellian $M:=\mathcal{M}_{[1,0,1]}(v)$,
\begin{equation}\label{Special form solution}
f^\v_l=M+\v^r g^\v_l \sqrt{M}, \quad \up{for} \quad l \in \{1,2\},
\end{equation}
and by substituting \eqref{Special form solution} into \eqref{Scaled BE equaion-1}, we obtain the remainder systems satisfied by $(g^\v_1,g^\v_2)$,
\begin{equation}\label{Scaled BE equation}
\begin{cases}
\v\p_t g^\v_1+v\c\n_x g^\v_1+\v^{-c_1}\hat{L}g^\v_1+\v\hat{L}(g^\v_1,g^\v_2)=\v^{r-c_1}\hat{\Gamma}(g^\v_1,g^\v_1)+\v^{2r+1-c_1} \hat
{\Gamma}(g^\v_1,g^\v_2),\\[4pt]
\v\p_t g^\v_2+v\c\n_x g^\v_2+\v^{-c_2}\hat{L}g^\v_2+\v\hat{L}(g^\v_2,g^\v_1)=\v^{r-c_2}\hat{\Gamma}(g^\v_2,g^\v_2)+\v^{2r+1-c_2} \hat{\Gamma}(g^\v_2,g^\v_1),\\[4pt]
\end{cases}
\end{equation}
where the linearized Boltzmann operator $\hat{L}$ and the bilinear symmetric operator $\hat{\Gamma}$ are given by
\begin{equation}\label{Linear operators}
\begin{aligned}
\hat{L}(g_l,g_n):= & -\frac{1}{\sqrt{M}}\big[Q(g_l \sqrt{M},M)+Q(M,g_n \sqrt{M})\big], \quad \hat{L}g_l := \hat{L}(g_l,g_l),\\
\hat{\Gamma}(g_l,g_n):= & \frac{1}{\sqrt{M}}Q(g_l \sqrt{M},g_n \sqrt{M}),
\end{aligned}
\end{equation}
for $l,n \in \{1,2\}$. Note that \cite{GFSRL05}, for $i = 1,2,3$,
\begin{equation}\label{Ker of L}
\text{Ker}\hat{L} = \text{Span}\{\sqrt{M},\, v_i\sqrt{M}, \,|v|^2\sqrt{M}\},
\end{equation}
therefore, denote the vectors $A=(A_i)$ and tensors $B=(B_{ij})$ with $i,j = 1,2,3$ as follows:
\begin{equation}\label{The form of A and B}
A_i(v)=\frac{1}{2} \left(|v|^2-5\right) v_i \sqrt{M}, \quad B_{ij}(v) = \left(v_iv_j-\frac{1}{3}|v|^2\delta_{ij}\right)\sqrt{M},
\end{equation}
then $A_i,B_{ij} \in \text{Ker}^{\perp}\hat{L}$. Also, $\hat{A}=(\hat{A}_i)$ and $\hat{B}=(\hat{B}_{ij})$ are the unique $\hat{L}^{-1}$ of $A,B$ in $\text{Ker}^{\perp}\hat{L}$, i.e.,
\begin{equation}\label{The form of hA and hB}
\hat{L}(\hat{A}_i)=A_i, \quad \hat{L}(\hat{B}_{ij})=B_{ij}.
\end{equation}

Our first main theorem is about the hydrodynamic systems that can be derived from the BEGM equation \eqref{Scaled BE equation} with various scalings.

\begin{theorem}\label{Main-theorem-1}
For $l,n \in \{1,2\}$, let $f^\v_l$ be a sequence of nonnegative solutions to the scaled BEGM equation \eqref{Scaled BE equaion-1} in the form of \eqref{Special form solution}. Assume the sequence $g^\v_l$ converges to $g_l$ and the following moments
\begin{equation*}
\begin{aligned}
&\l g^\v_l,  \sqrt{M}\r_v, \quad \l g^\v_l,  v\sqrt{M}\r_v, \quad \l g^\v_l,v\sqrt{M}\otimes v\r_v, \quad \l g^\v_l,v|v|^2\sqrt{M}\r_v,\\
&\l g^\v_l, \hat{A}(v)\otimes v\sqrt{M}\r_v, \quad \l \hat{\Gamma}(g^\v_l,g^\v_n),\hat{A}(v)\sqrt{M}\r_v,\\
&\l g^\v_l, \hat{B}(v)\otimes v\sqrt{M}\r_v, \quad \l \hat{\Gamma}(g^\v_l,g^\v_n),\hat{B}(v)\sqrt{M}\r_v
\end{aligned}
\end{equation*}
converge to
\begin{equation*}
\begin{aligned}
&\l g_l,\sqrt{M}\r_v, \quad \l g_l,v\sqrt{M}\r_v, \quad \l g_l,v\otimes v\sqrt{M}\r_v, \quad \l g_l,v|v|^2\sqrt{M}\r_v, \\
&\l g_l,\hat{A}(v)\otimes v\sqrt{M}\r_v, \quad \l \hat{\Gamma}(g_l,g_n), \hat{A}(v)\sqrt{M}\r_v,\\
&\l g_l,\hat{B}(v)\otimes v\sqrt{M}\r_v, \quad \l \hat{\Gamma}(g_l,g_n),\hat{B}(v)\sqrt{M}\r_v
\end{aligned}
\end{equation*}
in the sense of distribution as $\v \to 0$. \\
Then, $g_l$ has the form
\begin{equation}\label{The form of g}
g_l(t,x,v) = \left[ \rho_l(t,x) + u_l(t,x) \c v + \frac{1}{2}\left(|v|^2-3\right) \theta_l(t,x)\right] \sqrt{M},
\end{equation}
where $u_l$, $\rho_l$ and $\theta_l$ satisfy the divergence-free and Boussinesq relation
\begin{equation}\label{The Boussinesq relation}
\up{div}_x u_l=0, \quad \n_x(\rho_l+\theta_l)=0,
\end{equation}
and are solutions to the following equations: for $l,n \in \{1,2\}$ and $l\neq n$,
\begin{itemize}
\item For $r=1,\,c_l=c_n=1$,
\begin{equation}\label{The limits of Eqs-3}
\begin{cases}
\p_t u_l+u_l\c\n_x u_l+\frac{1}{\sigma}(u_l-u_n)+\n_x p_l=\mu\Delta_x u_l,\\[4pt]
\p_t u_n+u_n\c\n_x u_n+\frac{1}{\sigma}(u_n-u_l)+\n_x p_n=\mu\Delta_x u_n,\\[4pt]
\p_t\theta_l+u_l\c\n_x\theta_l+\frac{1}{\lambda}(\theta_l-\theta_n)=\kappa\Delta_x\theta_l,\\[4pt]
\p_t\theta_n+u_n\c\n_x\theta_n+\frac{1}{\lambda}(\theta_n-\theta_l)=\kappa\Delta_x\theta_n.\\[4pt]
\end{cases}
\end{equation}

\item For $r=1,\, 1 < c_l < 2,\, c_n=1$
\begin{equation}\label{The limits of Eqs-2}
\begin{cases}
\p_t u_l+u_l\c\n_x u_l+\frac{1}{\sigma}(u_l-u_n)+\n_x p_l=0,\\[4pt]
\p_t u_n+u_n\c\n_x u_n+\frac{1}{\sigma}(u_n-u_l)+\n_x p_n=\mu\Delta_x u_n,\\
\p_t\theta_l+u_l\c\n_x \theta_l+\frac{1}{\lambda}(\theta_l-\theta_n)=0,\\[4pt]
\p_t\theta_n+u_n\c\n_x \theta_n+\frac{1}{\lambda}(\theta_n-\theta_l)=\kappa\Delta_x \theta_n.\\[4pt]
\end{cases}
\end{equation}

\item For $r=1,\,1 < c_l=c_n < 2$,
\begin{equation}\label{The limits of Eqs-5}
\begin{cases}
\p_t u_l+u_l\c\n_x u_l+\frac{1}{\sigma}(u_l-u_n)+\n_x p_l=0,\\[4pt]
\p_t u_n+u_n\c\n_x u_n+\frac{1}{\sigma}(u_n-u_l)+\n_x p_n=0,\\
\p_t\theta_l+u_l\c\n_x\theta_l+\frac{1}{\lambda}(\theta_l-\theta_n)=0,\\[4pt]
\p_t\theta_n+u_n\c\n_x\theta_n+\frac{1}{\lambda}(\theta_n-\theta_l)=0.\\[4pt]
\end{cases}
\end{equation}

\item For $r>1,\,c_l=c_n=1$,
\begin{equation}\label{The limits of Eqs-4}
\begin{cases}
\p_t u_l+\frac{1}{\sigma}(u_l-u_n)+\n_x p_l=\mu\Delta_x u_l,\\[4pt]
\p_t u_n+\frac{1}{\sigma}(u_n-u_l)+\n_x p_n=\mu\Delta_x u_n,\\
\p_t\theta_l+\frac{1}{\lambda}(\theta_l-\theta_n)=\kappa\Delta_x\theta_l,\\[4pt]
\p_t\theta_n+\frac{1}{\lambda}(\theta_n-\theta_l)=\kappa\Delta_x\theta_n.\\[4pt]
\end{cases}
\end{equation}

\item For $r>1,\, 1 < c_l < 2r, \,c_n=1$,
\begin{equation}\label{The limits of Eqs-1}
\begin{cases}
\p_t u_l+\frac{1}{\sigma}(u_l-u_n)+\n_x p_l=0, \\[4pt]
\p_t u_n+\frac{1}{\sigma}(u_n-u_l)+\n_x p_n=\mu\Delta_x u_n,\\
\p_t\theta_l+\frac{1}{\lambda}(\theta_l-\theta_n)=0,\\[4pt]
\p_t\theta_n+\frac{1}{\lambda}(\theta_n-\theta_l)=\kappa\Delta_x\theta_n.\\[4pt]
\end{cases}
\end{equation}

\item For $r>1,\,1 < c_l=c_n <2r$,
\begin{equation}\label{The limits of Eqs-6}
\begin{cases}
\p_t u_l+\frac{1}{\sigma}(u_l-u_n)+\n_x p_l=0,\\[4pt]
\p_t u_n+\frac{1}{\sigma}(u_n-u_l)+\n_x p_n=0,\\
\p_t\theta_l+\frac{1}{\lambda}(\theta_l-\theta_n)=0,\\[4pt]
\p_t\theta_n+\frac{1}{\lambda}(\theta_n-\theta_l)=0,
\end{cases}
\end{equation}
\end{itemize}
where $u_l,\,p_l,\,\theta_l$ are the velocity, pressure and temperature of the different fluids, the constants $\mu,\,\kappa,\,\sigma,\,\lambda$ are given in \eqref{The constants of nu and kappa} and \eqref{The constants of sigma and lambda}.
\end{theorem}

\begin{remark}
To the best of our knowledge, except the two-fluid incompressible Navier-Stokes-Fourier system \eqref{The limits of Eqs-3} that has been derived in \cite[Chapter 2]{Saint-Raymond19}, the other systems above are derived from the scaled BEGM equation \eqref{Scaled BE equation} for the first time.
\end{remark}

Then, in terms of a particular case derived in Theorem \ref{Main-theorem-1} above, namely \eqref{The limits of Eqs-3}, we will provide the rigorous justification of the well-posedness and the hydrodynamic limit process.

More precisely, selecting $c_1 = c_2 = 1$ in the scaled BEGM equation \eqref{Scaled BE equaion-1} leads to
\begin{equation}\label{Scaled gas mixture BE system-0}
\begin{cases}
\v\p_t f^\v_1+v\c\n_xf^\v_1=\frac{1}{\v}Q(f^\v_1,f^\v_1)+\v Q(f^\v_1,f^\v_2),\\[4pt]
\v\p_t f^\v_2+v\c\n_xf^\v_2=\frac{1}{\v}Q(f^\v_2,f^\v_2)+\v Q(f^\v_2,f^\v_1),\\[4pt]
\end{cases}
\end{equation}
if we choose $r = 1$ in the expansion form \eqref{Special form solution}, the reminder system satisfied by $(g^\v_1, g^\v_2)$ becomes
\begin{equation}\label{Scaled gas mixture BE system-0g}
\begin{cases}
\v\p_t g^\v_1+v\c\n_x g^\v_1+\v^{-1}\hat{L}g^\v_1+\v\hat{L}(g^\v_1,g^\v_2)=\v^{r-1}\hat{\Gamma}(g^\v_1,g^\v_1)+\v^{2r} \hat
{\Gamma}(g^\v_1,g^\v_2),\\[4pt]
\v\p_t g^\v_2+v\c\n_x g^\v_2+\v^{-1}\hat{L}g^\v_2+\v\hat{L}(g^\v_2,g^\v_1)=\v^{r-1}\hat{\Gamma}(g^\v_2,g^\v_2)+\v^{2r} \hat{\Gamma}(g^\v_2,g^\v_1),
\end{cases}
\end{equation}
with initial data
\begin{equation}\label{Initial data-0g}
\left( g^{\v,in}_1(x,v),g^{\v,in}_2(x,v) \right)^\top
\end{equation}
such that
\begin{equation}\label{initial-f}
    f^{\v,in}_l(x,v) = M+\v g^{\v,in}_l \sqrt{M} \geq 0, \quad \text{for} \quad l \in \{1,2\},
\end{equation}

For brevity, we denote distribution functions $\bff^\v,\,\bg^\v$ as the vector form
\[
\bff^\v: =(f^\v_1(t,x,v),f^\v_2(t,x,v))^\top \quad \up{and} \quad \bg^\v := (g^\v_1(t,x,v), g^\v_2(t,x,v))^\top,
\]
and also denote the macroscopic functions $\bu,\,\btheta,\,\brho$ as
\[
\bu=(u_1(t,x),u_2(t,x))^\top \quad \up{and} \quad \btheta=(\theta_1(t,x),\theta_2(t,x))^\top \quad \up{and} \quad \brho=(\rho_1(t,x),\rho_2(t,x))^\top.
\]
The initial data $\bff^{\v,in},\,\bg^{\v,in}$ and $\bu^{in},\,\btheta^{in},\,\brho^{in}$ are defined following the similar manner.

The following theorem presents the well-posedness of $\bff^\v$ to the scaled BEGM equation \eqref{Scaled gas mixture BE system-0} around the global Maxwellian $M$. By considering the expansion form \eqref{Special form solution}, this is equivalent to prove the global well-posedness of $\bg^\v$ in \eqref{Scaled gas mixture BE system-0g}.

\begin{theorem}\label{Global-in-time solution of BE}
For any integer $s\geq 3$, there exists small $\v_0, l_0 > 0 $ such that for any $0<\v\leq \v_0$, if $\mathbb{E}_s(0)\leq l_0$, the Cauchy problem \eqref{Scaled gas mixture BE system-0g}-\eqref{Initial data-0g} admits a unique solution $\bg^\v=(g^\v_1,\,g^\v_2)$ satisfying
\begin{equation}\label{Solution in space}
\bg^\v \in L^\infty([0,+\infty);H^s_x L^2_v), \quad \P^\perp \bg^\v , \, \mathbf{P}^\perp \bg^\v \in L^2([0,+\infty);H^s_x L^2_v(\nu))
\end{equation}
with uniform energy estimate
\begin{equation}\label{Uniform energy estimate}
\sup_{t\geq 0}\mathbb{E}_s(t) + \tilde{C} \int_0^{+\infty}\mathbb{D}_s(\tau) \,d\tau \lesssim \mathbb{E}_s(0),
\end{equation}
where the projection operators $\P^\perp, \,\mathbf{P}^\perp$ are defined in \eqref{Fluid part of g with L-1}-\eqref{Fluid part of g with L-2}, the energy and dissipation functionals are defined in \eqref{The energy functional}-\eqref{The energy dissipative functional}, and the constants $l_0$, $\tilde{C}$ are independent of $\v$.
\end{theorem}

Thanks to the well-posedness and uniform energy estimate obtained in Theorem \ref{Global-in-time solution of BE} above, we can rigorously justify the first type of hydrodynamical limit derived in Theorem \ref{Main-theorem-1}, i.e., from the scaled BEGM equation \eqref{Scaled gas mixture BE system-0g} to the two-fluid incompressible Navier-Stokes-Fourier system \eqref{The limits of Eqs-3}.

Generally speaking, as $\v \to 0$, the solution $(g^\v_1, g^\v_2)$ to \eqref{Scaled gas mixture BE system-0g} can be proved to converge to the function constructed by the solutions to the following incompressible Navier-Stokes-Fourier system:
\begin{equation}\label{The Boltzmann equation for gaseous mixtures NSF equ}
\begin{cases}
\p_t u_1+u_1\c\n_x u_1+\frac{1}{\sigma}(u_1-u_2)+\n_x p_1=\mu\Delta_x u_1,\\[4pt]
\p_t u_2+u_2\c\n_x u_2+\frac{1}{\sigma}(u_2-u_1)+\n_x p_2=\mu\Delta_x u_2,\\[4pt]
\p_t\theta_1+u_1\c\n_x\theta_1+\frac{1}{\lambda}(\theta_1-\theta_2)=\kappa\Delta_x\theta_1,\\[4pt]
\p_t\theta_2+u_2\c\n_x\theta_2+\frac{1}{\lambda}(\theta_2-\theta_1)=\kappa\Delta_x\theta_2,\\[4pt]
\up{div}_x u_1=\up{div}_x u_2=0,\\[4pt]
\n_x(\rho_1+\theta_1) = \n_x(\rho_2+\theta_2) =0.
\end{cases}
\end{equation}

Now, we are in a position to specifically state the theorem concerning the hydrodynamic limit.

\begin{theorem}\label{Limit of Fluid equations}
Under the same assumptions and parameters $s,\,\v_0,\,l_0$ in Theorem \ref{Global-in-time solution of BE}, and let $\brho_0(x),\,\bu_0(x)$ and $\btheta_0(x)$ be given such that
\begin{equation}
\bg^{\v,in}(x,v) \to \left[\brho_0(x) + \bu_0(x) \c v + \btheta_0(x) \left(\frac{|v|^2}{2}-\frac{3}{2}\right) \right] \M
\end{equation}
strongly in $H^s_x L^2_v$ as $\v\to 0$.
Let $\bg^\v$ be a sequence of solutions to the scaled BEGM equation \eqref{Scaled gas mixture BE system-1} constructed in Theorem \ref{Global-in-time solution of BE}. Then,
\begin{equation}
\bg^\v(t,x,v) \to \left[ \bu(t,x) \c v + \btheta(t,x) \left(\frac{|v|^2}{2}-\frac{5}{2}\right) \right] \M,
\end{equation}
in the sense of weak-$\star$ in $t\geq 0$, strong in $H^{s-\eta}_x$, and weak in $L^2_v$, where $\bu$ and $\btheta$ are the solution to the incompressible Navier-Stokes-Fourier system \eqref{The Boltzmann equation for gaseous mixtures NSF equ} with initial data
\begin{equation}
\bu^{in}(x) = \bu_0(x), \quad \btheta^{in}(x) = \frac{3}{5} \btheta_0(x)-\frac{2}{5}\brho_0(x),
\end{equation}
Furthermore, the convergence of the moments holds: for any $\eta>0$,
\begin{equation}
\begin{aligned}
&\mathcal{P} \l\bg^\v, v\M\r_v \to \bu, \quad \text{in} \quad C([0,+\infty);H^{s-1-\eta}_x) \cap L^\infty([0,+\infty);H^{s-1}_x),\\[4pt]
&\l\bg^\v, \left(\frac{|v|^2}{5}-1\right) \M\r_v \to \btheta, \quad \text{in} \quad C([0,+\infty);H^{s-1-\eta}_x)\cap L^\infty([0,+\infty);H^{s-1}_x),
\end{aligned}
\end{equation}
with the Leray projection $\mathcal{P}$.
\end{theorem}

The rest of the paper is organized as follows: in Section \ref{sec:Formal-limit}, we present the derivation of various hydrodynamic systems via formal analysis (Theorem \ref{Main-theorem-1}). 
The global well-posedness (Theorem \ref{Global-in-time solution of BE}) of $\bg^\v$ is proved with the help of the uniform energy estimate in Section \ref{sec:Global wellposedness}.
Finally, the hydrodynamic limiting process (Theorem \ref{Limit of Fluid equations}) will be rigorously justified in Section \ref{sec:limit}.



\section{Formal derivation (Theorem \ref{Main-theorem-1})}
\label{sec:Formal-limit}

In this section, we present the specific formal derivation of the systems in Theorem \ref{Main-theorem-1}.

\subsection{Preliminary properties}
\label{subsec:preliminary}

For $A, B$ and $\hat{A}, \hat{B}$ in \eqref{The form of A and B}-\eqref{The form of hA and hB}, we have the following property:
\begin{lemma}{\cite{BCGFLCD93, DLGF94}}
\label{A and B}
There exist two scalar positive functions $\a$ and $\b$ such that
\begin{equation}\label{Definition of hat A and B}
\hat{A}(v)=\a(|v|)A(v), \quad \hat{B}(v)=\b(|v|)B(v).
\end{equation}
Furthermore, we have
\begin{equation*}
\begin{aligned}
\l\hat{A}_i,A_j \r_v&=\frac{5}{2}\kappa\delta_{ij},\\
\l\hat{B}_{ij},B_{kl} \r_v&=\mu(\delta_{ik}\delta_{jl}+\delta_{il}\delta_{jk}-\frac{2}{3}\delta_{ij}\delta_{kl}),\\
\end{aligned}
\end{equation*}
where $\delta_{ij}$ is the Kronecker delta function, and the constants $\mu,\,\kappa > 0$ are given by
\begin{equation}\label{The constants of nu and kappa}
\begin{aligned}
\kappa=&\l\a(|v|),A(v)\otimes A(v)\r_v=\frac{2}{15\sqrt{2\pi}}\int_0^\infty \a (r)r^6 \up{e}^{-\frac{r^2}{2}} \,dr,\\
\mu=&\l\b(|v|),B(v)\otimes B(v)\r_v=\frac{1}{6\sqrt{2\pi}}\int_0^\infty \b(r)(r^2-5)^2 r^4\up{e}^{-\frac{r^2}{2}} \,dr.
\end{aligned}
\end{equation}
\end{lemma}

For operators $\hat{L}$ and $\hat{\Gamma}$ in \eqref{Linear operators}, we have
\begin{lemma}{\cite{BCGFLD91}}
\label{Poposition of L and Q}
The following properties hold for the linearized Boltzmann operator $\hat{L}$ and the bilinear symmetric operator $\Gamma$:

\up{(i)} The linearized Boltzmann operator $\hat{L}$ is self-adjointness in $L^2_v$, i.e. for any $f,g\in L^2_v$,
\begin{equation}\label{self-adjointness}
\l \hat{L}f,g\r_v=\l f,\hat{L}g\r_v.
\end{equation}

\up{(ii)} For any $g \in \text{Ker}\hat{L}$,
\begin{equation}\label{Property of Q}
\frac{1}{2}\hat{L} (g^2)=\hat{\Gamma}(g,g).
\end{equation}
\end{lemma}


\subsection{Formal derivation of the hydrodynamic systems}
\label{subsec:formal}

Now, we are in a position to demonstrate the derivation of the hydrodynamic systems listed in the Theorem \ref{Main-theorem-1}. To make it clear, this derivation is divided into two steps: the first step focuses on deriving the common conditions, i.e., \eqref{The form of g}-\eqref{The Boussinesq relation}, whereas the second step involves deriving the unique equations that result from different selections of $r$ and $c_l$, corresponding to equations \eqref{The limits of Eqs-3} to \eqref{The limits of Eqs-6}.

\textbf{Step I: derivation of common conditions \eqref{The form of g}-\eqref{The Boussinesq relation}.} We rewrite the scaled BEGM equation \eqref{Scaled BE equation} as follows: for $l,n \in \{1,2\}$ and $l\neq n$,
\begin{equation}\label{Scaled BE equation-2}
\hat{L} g^\v_l = -\v^{1+c_l}\p_t g^\v_l-\v^{c_l} v\c\n_x g^\v_l - \v^{c_l+1}\hat{L}(g^\v_l,g^\v_n) + \v^{r}\hat{\Gamma}(g^\v_l,g^\v_l) + \v^{2r+1} \hat{\Gamma}(g^\v_l,g^\v_n).
\end{equation}
As only the cases of $r \geq 1$ and $1 \leq c_l < 2r$ are considered in Theorem \ref{Main-theorem-1}, by letting $\v \to 0$ in \eqref{Scaled BE equation-2} above and combining the assumption of moment convergence, we have, for $l \in \{1,2\}$,
\begin{equation*}\label{Lg}
\hat{L} g_l = 0,
\end{equation*}
which implies that $g_l$ belongs to $\up{Ker}\hat{L}$ as in \eqref{Ker of L} and can be written in the form of \eqref{The form of g}.

Furthermore, the Boussinesq relation \eqref{The Boussinesq relation} directly follows from the conservation of mass and momentum: for $l,n \in \{1,2\}$ and $l\neq n$,
\begin{equation}\label{BE equation-mass,momentum-1}
\begin{cases}
\v\p_t\l g^\v_l, \sqrt{M}\r_v+\up{div}_x\l g^\v_l, v\sqrt{M}\r_v+\v\l\hat{L}(g^\v_l,g^\v_n),\sqrt{M}\r_v=\v^{2r+1-c_l}\l\hat{\Gamma}(g^\v_l,g^\v_n),\sqrt{M}\r_v,\\[8pt]
\v\p_t\l g^\v_l, v\sqrt{M}\r_v+\up{div}_x\l g^\v_l,v\otimes v\sqrt{M}\r_v+\v\l\hat{L}(g^\v_n,g^\v_l),v\sqrt{M}\r_v=\v^{2r+1-c_l}\l\hat{\Gamma}(g^\v_l,g^\v_n),v\sqrt{M}\r_v ,\\[4pt]
\end{cases}
\end{equation}
then, since $2r > c_l$, letting $\v \to 0$ above, we obtain
\begin{equation}
\up{div}_x\l g_l,v\sqrt{M}\r_v =0, \quad \up{div}_x\l g_l ,v\otimes v\sqrt{M}\r_v=0.
\end{equation}
By substituting $g_l$ in \eqref{The form of g} into the left-hand side above, we can obtain the Boussinesq condition in \eqref{The Boussinesq relation}.
Similarly, performing the same operation on energy conservation yields the divergence-free condition in \eqref{The Boussinesq relation}.

\textbf{Step II: derivation of the distinct equations \eqref{The limits of Eqs-3}-\eqref{The limits of Eqs-6}.} We rewrite \eqref{Scaled BE equation-2} to 
\begin{equation}\label{Scaled BE equation-3}
\p_t g^\v_l+\frac{1}{\v}v\c\n_x g^\v_l+\frac{1}{\v^{c_l+1}}\hat{L}g^\v_l+\v\hat{L}(g^\v_l,g^\v_n)=\v^{r-c_l-1}\hat{\Gamma}(g^\v_l,g^\v_l)+\v^{2r-c_l}\hat{\Gamma}(g^\v_l,g^\v_n),
\end{equation}
and further deduce that
\begin{equation}\label{The case of r=q-1}
\begin{cases}
&\p_t\l g^\v_l, v\sqrt{M}\r_v+\frac{1}{\v}\up{div}_x\l g^\v_l,v\otimes v\sqrt{M}\r_v+\l\hat{L}(g^\v_l,g^\v_n),v\sqrt{M}\r_v\\[6pt]
=&\v^{2r-c_l}\l\hat{\Gamma}(g^\v_l,g^\v_n),v\sqrt{M}\r_v,\\[8pt]
&\p_t\l g^\v_l,\, \left(\frac{|v|^2}{2}-\frac{5}{2}\right)\sqrt{M}\r_v+\frac{1}{\v}\up{div}_x\l g^\v_l,\, \left(\frac{|v|^2}{2}-\frac{5}{2}\right)v\sqrt{M}\r_v + \l\hat{L}(g^\v_l,g^\v_n),\left(\frac{|v|^2}{2}-\frac{5}{2}\right)\sqrt{M}\r_v\\[6pt]
=&\v^{2r-c_l}\l\hat{\Gamma}(g^\v_l,g^\v_n),\, \left(\frac{|v|^2}{2}-\frac{5}{2}\right)\sqrt{M}\r_v.
\end{cases}
\end{equation}

By a direct calculation involving \eqref{The form of g}, we have
\begin{equation*}\label{The limits of partial}
\begin{aligned}
\lim_{\v\to 0}\p_t\l g^\v_l, v\sqrt{M}\r_v&=\p_t u_l,\\[4pt]
\lim_{\v\to 0}\p_t\l g^\v_l, \left(\frac{|v|^2}{2}-\frac{5}{2}\right) \sqrt{M}\r_v&=\frac{5}{2}\p_t\theta_l.
\end{aligned}
\end{equation*}

According to \cite[Chapter 2]{Saint-Raymond19}, we obtain
\begin{equation*}\label{Linear operator of L}
\begin{aligned}
\lim_{\v\to 0}\l\hat{L}(g^\v_l,g^\v_n),v\sqrt{M}\r_v&=\frac{1}{\sigma}(u_l-u_n),\\[4pt]
\lim_{\v\to 0}\l\hat{L}(g^\v_l,g^\v_n),\left(\frac{|v|^2}{2}-\frac{5}{2}\right) \sqrt{M}\r_v&=\frac{5}{2\lambda}(\theta_l-\theta_n),
\end{aligned}
\end{equation*}
where the electrical conductivity $\sigma$ and the energy conductivity $\lambda$ are given by
\begin{equation}\label{The constants of sigma and lambda}
\begin{aligned}
\frac{1}{\sigma} = & \frac{1}{2}\int_{\R^3} v \hat{L}(v\sqrt{M},v\sqrt{M})\sqrt{M} \,dv,\\
\frac{1}{\lambda} = & \frac{1}{20}\int_{\R^3}|v|^2 \hat{L}(|v|^2\sqrt{M},|v|^2\sqrt{M})\sqrt{M} \,dv.
\end{aligned}
\end{equation}

Since $2r>c_l$, the terms $\v^{2r-c_l}\l\hat{\Gamma}(g^\v_l,g^\v_n),v\sqrt{M}\r_v$ and $\v^{2r-c_l}\l\hat{\Gamma}(g^\v_l,g^\v_n),(\frac{|v|^2}{2}-\frac{5}{2})\sqrt{M}\r_v$ on the right-hand side of \eqref{The case of r=q-1} vanish.
Hence, to complete the derivation of all equations in Theorem \ref{Main-theorem-1}, it needs to estimate $\frac{1}{\v}\up{div}_x\l g^\v_l,v\otimes v\sqrt{M}\r_v$ and $\frac{1}{\v}\up{div}_x\l g^\v_l,(\frac{|v|^2}{2}-\frac{5}{2})v\sqrt{M}\r_v$ in \eqref{The case of r=q-1}.

Applying the self-adjointness of $\hat{L}$ in lemma \ref{A and B}, we have
\begin{equation*}\label{The case of limits r=q-2}
\begin{aligned}
\lim_{\v\to 0}\mathcal{P}\left(\frac{1}{\v}\up{div}_x\l g^\v_l,v\otimes v\sqrt{M}\r_v\right) = \lim_{\v\to 0}\frac{1}{\v}\up{div}_x\l g^\v_l,B\r_v = &\lim_{\v\to 0}\frac{1}{\v}\up{div}_x\l g^\v_l,\hat{L}\hat{B}\r_v\\
=&\lim_{\v\to 0}\frac{1}{\v}\up{div}_x\l \hat{L}g^\v_l,\hat{B}\r_v,\\
\lim_{\v\to 0}\frac{1}{\v}\up{div}_x\l g^\v_l,\,\left(\frac{|v|^2}{2}-\frac{5}{2}\right)v\sqrt{M}\r_v = \lim_{\v\to 0}\frac{1}{\v}\up{div}_x\l g^\v_l,A\r_v & =\lim_{\v\to 0}\frac{1}{\v}\up{div}_x\l g^\v_l,\hat{L}\hat{A}\r_v\\
=&\lim_{\v\to 0}\frac{1}{\v}\up{div}_x\l\hat{L}g^\v_l,\hat{A}\r,
\end{aligned}
\end{equation*}
where $\mathcal{P}$ is the Leray projection.

Recalling the scaled BEGM equation \eqref{Scaled BE equation-3}
\begin{equation}\label{Scaled BE equation-4}
\frac{1}{\v}\hat{L} g^\v_l=\v^{r-1}\hat{\Gamma}(g^\v_l,g^\v_l)+\v^{2r}\hat{\Gamma}(g^\v_l,g^\v_n)-\v^{c_l}\p_t g^\v_l-\v^{c_l-1}v\c\n_x g^\v_l-\v^{c_l+1}\hat{L}(g^\v_l,g^\v_n),
\end{equation}
we have
\begin{equation*}
\begin{aligned}
&\lim_{\v\to 0}\mathcal{P}\frac{1}{\v}\up{div}_x\l g^\v_l,v\otimes v\sqrt{M}\r_v\\[4pt]
=&\lim_{\v\to 0}\v^{r-1}\up{div}_x\l \hat{\Gamma}(g^\v_l,g^\v_l),\hat{B}\r_v+\lim_{\v\to 0}\v^{2r}\up{div}_x\l\hat{\Gamma}(g^\v_l,g^\v_n),\hat{B}\r_v\\
&-\lim_{\v\to 0}\v^{c_l}\up{div}_x\l\p_t g^\v_l,\hat{B}\r_v-\lim_{\v\to 0}\v^{c_l-1}\up{div}_x\l v\c\n_x g^\v_l,\hat{B}\r_v\\
&-\lim_{\v\to 0}\v^{c_l+1}\up{div}_x\l\hat{L}(g^\v_l,g^\v_n),\hat{B}\r_v\\[4pt]
=&\lim_{\v\to 0}\v^{r-1}\up{div}_x\l \hat{\Gamma}(g^\v_l,g^\v_l),\hat{B}\r_v-\lim_{\v\to 0}\v^{c_l-1}\up{div}_x\l v\c\n_x g^\v_l,\hat{B}\r_v \,.\\
\end{aligned}
\end{equation*}

Therefore, the different choices of $r$ and $c_l$ will lead to the different terms:
\begin{itemize}
    \item If $r=1$, noting $\frac{1}{2}\hat{L}(g_l^2) = \hat{\Gamma}(g_l,g_l)$ for $g_l\in\up{Ker}\hat{L}$ as in Lemma \ref{Poposition of L and Q},
    \begin{equation*}
    \begin{aligned}
    \lim_{\v\to 0}\up{div}_x\l \hat{\Gamma}(g^\v_l,g^\v_l),\hat{B}\r_v=\frac{1}{2}\l \hat{g}_l^2,\hat{B}\r_v=\frac{1}{2}\l g_l^2,\hat{L}\hat{B}\r_v= \frac{1}{2}\l g_l^2,B\r_v=u_l\c\n_x u_l.
    \end{aligned}
    \end{equation*}

    \item If $r>1$,
$$\lim_{\v\to 0}\v^{r-1}\up{div}_x\l \hat{\Gamma}(g^\v_l,g^\v_l),\hat{B}\r_v=0.$$

    \item If $c_l=1$, applying Lemma \ref{Poposition of L and Q},
    \begin{equation*}
    \lim_{\v\to 0}\up{div}_x\l v\c\n_x g^\v_l,\hat{B}\r_v = \up{div}_x\l v\c\n_x g_l,\beta(|v|)B\r_v=\mu\Delta_x u_l,
    \end{equation*}
    where the constant $\mu$ is given by \eqref{The constants of nu and kappa}.

    \item If $c_l>1$,
    $$\lim_{\v\to 0}\v^{c_l-1}\up{div}_x\l v\c\n_x g^\v_l,\hat{B}\r_v=0.$$
    Applying a similar approach, we can obtain
    \begin{equation*}
    \begin{split}
    &\lim_{\v\to 0}\frac{1}{\v}\up{div}_x\l g^\v_l,\, \left(\frac{|v|^2}{2}-\frac{5}{2}\right)v\sqrt{M}\r_v \\[4pt]
    =&\lim_{\v\to 0}\v^{r-1}\up{div}_x\l \hat{\Gamma}(g^\v_l,g^\v_l),\hat{A}\r_v-\lim_{\v\to 0}\v^{c_l-1}\up{div}_x\l v\c\n_x g^\v_l,\hat{A}\r_v.
    \end{split}
    \end{equation*}
\end{itemize}

Thus, the equations in Theorem \ref{Main-theorem-1} are derived by involving the following relations:
\begin{multline*}
\lim_{\v\to 0}\frac{1}{\v}\up{div}_x\l g^\v_l, \left(\frac{|v|^2}{2}-\frac{5}{2}\right)v\sqrt{M}\r_v = \left\{
\begin{array}{ccc}
\displaystyle \frac{5}{2}u_l\c\n_x\theta_l-\frac{5}{2}\kappa\Delta_x \theta_l& \qquad \up{if}\,\, r=1 \quad \up{and} \quad c_l=1,\\[4pt]
\displaystyle \frac{5}{2}u_l\c\n_x\theta_l& \qquad \up{if}\,\, r=1 \quad \up{and}\quad c_l>1,\\[4pt]
\displaystyle -\frac{5}{2}\kappa\Delta_x \theta_l& \qquad \up{if}\,\, r>1 \quad \up{and}\quad c_l=1,\\[4pt]
\displaystyle 0& \qquad \up{if}\,\, r>1 \quad \up{and} \quad c_l>1.\\[4pt]
\end{array}
\right.
\end{multline*}


\section{Global well-posedness of the scaled BEGM equation (Theorem \ref{Global-in-time solution of BE})}
\label{sec:Global wellposedness}

\subsection{Basic Setup}

In this section, we will prove the scaled BEGM equation \eqref{Scaled gas mixture BE system-1} admits a unique global-in-time solution for all $\v\in(0,\v_0]$ with small $\v_0 > 0$. To this end, we first make some preparations.


Then, by letting
\begin{equation}\label{Special form solution1}
    f^\v_l=M+\v g^\v_l\sqrt{M}, \quad \text{for} \quad l \in \{1,2\},
\end{equation}
the reminder system satisfied by $\bg^\v = (g^\v_1, g^\v_2)$ of scaled BEGM equation \eqref{Scaled gas mixture BE system-0g} can be rewritten in the vector form:
\begin{equation}\label{Scaled gas mixture BE system-1}
\v\p_t \bg^\v + v \c \n_x \bg^\v + \left(\frac{1}{\v}-\v\right) L \bg^\v + \v \L \bg^\v = \Gamma(\bg^\v, \bg^\v) + \v^2 \tilde{\Gamma}(\bg^\v, \bg^\v)
\end{equation}
with initial data
\begin{equation}\label{Initial data}
\bg^{\v,in}(x,v) = \left( g^{\v,in}_1(x,v),g^{\v,in}_2(x,v) \right)^\top
\end{equation}
satisfying \eqref{initial-f},
where the linearized collision operators $L,\,\mathcal{L}$ and the nonlinear collision operators $\Gamma,\,\tilde{\Gamma}$ are given in the vector form: for any $\bg=(g_1,g_2)^\top$ and $\bh=(h_1,h_2)^\top$,
\begin{equation}\label{Linear and nonlinear opeartors}
\begin{aligned}
L\bg= &(\hat{L}g_1,\hat{L}g_2)^\top,\\
=&-\frac{1}{\sqrt{M}}\Big(Q(M,g_1 \sqrt{M})+Q(g_1 \sqrt{M},M),Q(M,g_2 \sqrt{M})+Q(g_2 \sqrt{M},M)\Big)^\top,\\[6pt]
\L\bg=&(\hat{L}(g_1,g_2) + \hat{L}g_1, \hat{L}(g_2,g_1) + \hat{L}g_2)^\top,\\
=&-\frac{1}{\sqrt{M}}\Big(2Q(g_1\sqrt{M},M)+Q(M,\{g_1+g_2\}\sqrt{M}),2Q(g_2\sqrt{M},M)+Q(M,\{g_1+g_2\}\sqrt{M})\Big)^\top,\\[6pt]
\Gamma(\bg,\bh)=&(\hat{\Gamma}(g_1,h_1),\hat{\Gamma}(g_2,h_2))^\top,\\
=&\frac{1}{\sqrt{M}}\Big(Q(g_1 \sqrt{M},h_1 \sqrt{M}),Q(g_2 \sqrt{M},h_2 \sqrt{M})\Big)^\top,\\[6pt]
\tilde{\Gamma}(\bg,\bh)=&(\hat{\Gamma}(g_1,h_2),\hat{\Gamma}(g_2,h_1))^\top,\\
=&\frac{1}{\sqrt{M}}\Big(Q(g_1 \sqrt{M},h_2 \sqrt{M}),Q(g_2 \sqrt{M},h_1 \sqrt{M})\Big)^\top,
\end{aligned}
\end{equation}
with $\hat{L}$ and $\hat{\Gamma}$ being defined in \eqref{Linear operators}.

According to \cite{GuoYan03,GY06}, the kernel of $L$ and $\mathcal{L}$ are given by, for $i = 1,2,3$,
\begin{equation}\label{Space of Ker}
\begin{aligned}
\up{Ker}L :=& \up{Span}\left\{\psi_1(v),\,\psi_2(v),\,\psi_{i+2}(v),\,\psi_{i+5}(v),\, \psi_9(v),\,\psi_{10}(v)\right\},\\[4pt]
\up{Ker}\L :=& \up{Span}\left\{\phi_1(v),\,\phi_2(v),\,\phi_{i+2}(v),\,\phi_6(v)\right\},
\end{aligned}
\end{equation}
where
\begin{equation*}
\begin{split}
    \psi_1(v) =& [1,0]^\top \M, \quad \psi_2(v)=[0,1]^\top \M, \quad \psi_{i+2}(v)=[v_i,0]^\top \M, \quad \psi_{i+5}(v)=[0,v_i]^\top \M, \\[4pt]
    \psi_9(v) =& \left[ \frac{|v|^2}{3}-1,0 \right]^\top \M, \quad \psi_{10}(v)=\left[ 0,\frac{|v|^2}{3}-1 \right]^\top \M, \\[4pt]
    \phi_1(v) =& \psi_1(v),\quad \phi_2(v)=\psi_2(v), \quad \phi_{i+2}(v)=\psi_{i+2}(v)+\psi_{i+5}(v),\\[4pt]
    \phi_6(v) =& \frac{1}{2}\left[ |v|^2-3,|v|^2-3 \right]^\top \M.
\end{split}
\end{equation*}

We define the projection $\mathbb{P}: L^2_v \mapsto \up{Ker}\L$ in the sense that, for any $\bg \in L^2_v$,
\begin{equation}\label{Fluid part of g with L-1}
\begin{aligned}
\mathbb{P} \bg = & \l \bg,\phi_1\r_v \,\phi_1 + \l \bg,\phi_2\r_v \,\phi_2 + \sum_{i=1}^3\frac{1}{2}\l \bg, \phi_{i+2}\r_v \,\phi_{i+2} + \frac{1}{3}\l \bg, \phi_6\r_v \,\phi_6\\
=&\rho^1\phi_1 + \rho^2\phi_2 + \sum_{i=1}^3 u_i\phi_{i+2} + \theta\phi_6
\end{aligned}
\end{equation}
with the coefficients $\rho^1=\l \bg,\phi_1\r_v,\,\rho^2=\l \bg, \phi_2\r_v$, $u_i=\frac{1}{2}\l \bg,\phi_{i+2}\r_v$ for $i=1,2,3$, and $\theta=\frac{1}{3}\l \bg,\phi_6\r_v$.
Hence, $\bg$ can be decomposed into
\begin{equation}\label{decom1}
    \bg = \mathbb{P}\bg + (\mathbb{I}-\mathbb{P})\bg = \mathbb{P}\bg + \mathbb{P}^{\perp}\bg.
\end{equation}
In addition, we can define another projection $\mathbf{P}: L^2_v \mapsto \up{Ker} L$ that
\begin{equation}\label{Fluid part of g with L-2}
\begin{aligned}
\mathbf{P} \bg =&\Big(\tilde{a}_1(t,x)\M+\sum_{j=1}^3 \tilde{b}_{1j}(t,x)v_j\M+\tilde{c}_1(t,x)|v|^2\M,\\
&\tilde{a}_2(t,x)\M+\sum_{j=1}^3 \tilde{b}_{2j}(t,x)v_j\M+\tilde{c}_2(t,x)|v|^2\M\Big)^\top,
\end{aligned}
\end{equation}
where the coefficients $\tilde{a}_l(t,x)$, $\tilde{b}_{lk}(t,x)$ and $\tilde{c}_l(t,x)$ are given as in \cite{GuoYan03}:
\begin{equation*}\label{The constants of a,b and c}
\begin{aligned}
\tilde{a}_l(t,x) =&\int_{\R^3}\left( \frac{5}{2}-\frac{|v|^2}{2} \right) g_l \M \,dv, \quad \tilde{b}_{lk}(t,x) =\int_{\R^3}v_k g_l\M \,dv,\\
\tilde{c}_l(t,x) =&\int_{\R^3}\left( \frac{|v|^2}{6}-\frac{1}{2}\right) g_l \M \,dv
\end{aligned}
\end{equation*}
with $l=1,2$ and $k=1,2,3$. Therefore, $\bg$ can be alternatively decomposed into
\begin{equation}\label{decom2}
    \bg = \mathbf{P}\bg + (\mathbf{I}-\mathbf{P})\bg = \mathbf{P}\bg + \mathbf{P}^{\perp}\bg.
\end{equation}

Now, we introduce the energy functional
\begin{equation}\label{The energy functional}
\begin{aligned}
\mathbb{E}_s(t)=\|\bg^\v\|^2_{H^s_x L^2_v}:=\sum_{l=1}^2\|g^\v_l\|_{H^s_xL^2_v}^2,
\end{aligned}
\end{equation}
and the dissipation functional
\begin{equation}\label{The energy dissipative functional}
\begin{aligned}
\mathbb{D}_s(t)=\frac{1}{\v^2}\|\mathbf{P}^{\perp}\bg^\v\|_{H^s_xL^2_v(\nu)}^2 + \v\|\mathbb{P}^{\perp}\bg^\v\|^2_{H^s_x L^2_v(\nu)} + \|\n_x\mathbb{P}\bg^\v\|_{H^{s-1}_x}^2.
\end{aligned}
\end{equation}
And the initial energy and dissipation functionals follow that
\begin{equation}\label{Initial Energy}
\begin{aligned}
\mathbb{E}_s(0)=&\|\bg^{\v,in}\|^2_{H^s_x L^2_v} = \sum_{l=1}^2\|g_l^{\v,in}\|_{H^s_xL^2_v}^2,\\
\mathbb{D}_s(0)=&\frac{1}{\v^2}\|\mathbf{P}^{\perp}\bg^{\v,in}\|_{H^s_xL^2_v(\nu)}^2 + \v\|\mathbb{P}^{\perp}\bg^{\v,in}\|^2_{H^s_x L^2_v(\nu)}+\|\n_x\mathbb{P}\bg^{\v,in}\|_{H^{s-1}_x}^2.
\end{aligned}
\end{equation}

\subsection{Uniform energy estimates}
\label{subsec:apriori}

In this section, we will prove the uniform energy estimates for $\bg^{\v}$ to the system \eqref{Scaled gas mixture BE system-1}-\eqref{Initial data}.

\subsubsection{Preliminary results}
\label{subsubsec:preliminary}

The essential building block in obtaining uniform estimates is the dissipative structure provided by the linearized operators $L,\,\mathcal{L}$ defined in \eqref{Linear and nonlinear opeartors}.

According to \cite[Proposition 5.7]{Saint-Raymond19} and \cite[Lemma 3.2, Lemma 3.3]{GY06}, we have the following coercivity of the linearized collision operators $L,\,\mathcal{L}$:
\begin{lemma}\label{The coercivity of the linearized collision operators}
For any $\bg=(g_1,g_2)$, there exists $\delta>0$ such that
\begin{equation}\label{The coercivity of L-1}
\l L\bg, \bg \r_v \geq \delta\|\mathbf{P}^{\perp}\bg\|_{L^2_v(\nu)}^2,
\end{equation}
and
\begin{equation}\label{The coercivity of L-2}
\l \L \bg, \bg\r_v\geq \delta\|\mathbb{P}^{\perp} \bg\|_{L^2_v(\nu)}^2.
\end{equation}
\end{lemma}

In addition, thanks to \cite[Lemma 3.3]{GY06}, the following lemma is satisfied by the nonlinear operators $\Gamma,\,\tilde{\Gamma}$:
\begin{lemma}\label{The estimate of the symmetric operators}
For any $\bg=(g_1,g_2),\bh=(h_1,h_2),\bff=(f_1,f_2)$, then
\begin{equation}\label{The coercivity of Gamma}
\Big|\l\p_x^\a\Gamma(\bg,\bh),\p_x^\a \bff\r_v\Big| \lesssim \sum_{|\a_1|+|\a_2| \leq |\a|}\Big[\|\p_x^{\a_1}\bg\|_{L^2_v}\|\p_x^{\a_2}\bh\|_{L^2_v(\nu)}+\|\p_x^{\a_1}\bg\|_{L^2_v}\|\p_x^{\a_2}\bh\|_{L^2_v(\nu)}\Big]\|\p_x^\a\bff\|_{L^2_v(\nu)},
\end{equation}
and
\begin{equation}\label{The coercivity of tilde-Gamma}
\Big|\l\p_x^\a\tilde{\Gamma}(\bg,\bh),\p_x^\a\bff\r_v\Big|\lesssim\sum_{|\a_1|+|\a_2|\leq |\a|}\Big[\|\p_x^{\a_1}\bg\|_{L^2_v}\|\p_x^{\a_2}\bh\|_{L^2_v(\nu)}+\|\p_x^{\a_1}\bg\|_{L^2_v}\|\p_x^{\a_2}\bh\|_{L^2_v(\nu)}\Big]\|\p_x^\a\bff\|_{L^2_v(\nu)}.
\end{equation}
\end{lemma}

The local well-posedness of the scaled BEGM equation \eqref{Scaled gas mixture BE system-1} has been shown in \cite{JL22}.
\begin{proposition}{\cite[Proposition 3.1]{JL22}}\label{Lmm-Local}
For $s\geq 3$, there exists $T^*>0,\v_0>0$ and $l_0>0$ independent
of $\v$, such that for any $0<\v\leq\v_0$, $\mathbb{E}(0)\leq l_0$ and $T^*\leq\sqrt{l_0}$, the Cauchy problem \eqref{Scaled gas mixture BE system-1}-\eqref{Initial data} admits a unique solution $\bg^\v$ satisfying
\begin{equation*}
\bg^\v \in L^\infty (0,T^*; H^s_{x}L^2_v), \quad \P^\perp \bg^\v, \, \mathbf{P}\bg^\v \in L^2(0,T^*; H^s_{x}L^2_v(\nu))
\end{equation*}
with the energy estimate:
\begin{equation}\label{The energy bound}
\sup_{0\leq t\leq T^*}\mathbb{E}_s(t) + \int_0^{T^*} \left(\frac{1}{\v^2}\|\P^\perp \bg^\v\|_{H^s_x L^2_v(\nu)}^2+\|\mathbf{P}^\perp \bg^\v\|_{H^s_x L^2_v(\nu)}^2 \right) \,dt \lesssim l_0.
\end{equation}
\end{proposition}



\subsubsection{Energy estimates on microscopic parts: kinetic dissipation}
\label{subsubsec:kinetic-dissipations}

In this part, we first establish the energy estimates on the microscopic parts $\bg^\v$ by taking advantage of the kinetic dissipation.

\begin{lemma}\label{Pure spatial derivative estimate}
Under the assumptions of Theorem \ref{Global-in-time solution of BE}, let $\bg^\v$ be the solution to the system \eqref{Scaled gas mixture BE system-1}-\eqref{Initial data}, then, for $0< \v \leq 1$, there exists constant $C_1 > 0$ independent of $\v$ such that
\begin{equation}\label{Pure spatial esti}
\frac{1}{2}\frac{d}{dt}\|\bg^\v\|_{H^s_x L^2_v}^2 + \frac{\delta}{\v^2}\|\mathbf{P}^{\perp}\bg^\v\|_{H^s_x L^2_v(\nu)}^2   +\delta\|\mathbb{P}^{\perp}\bg^\v\|_{H^s_x L^2_v(\nu)}^2 \leq C_1 \mathcal{E}^{\frac{1}{2}}_s(t)\mathbb{D}_s(t),
\end{equation}
where $\delta$ is given in Lemma \ref{The coercivity of the linearized collision operators}.
\end{lemma}

\begin{proof}
Firstly, we rewrite the scaled BEGM equation \eqref{Scaled gas mixture BE system-1} as follows:
\begin{equation}\label{Scaled Boltzmann equation for gaseous mixtures BE equation-1}
\p_t \bg^\v + \frac{1}{\v} v\c\n_x \bg^\v + \left(\frac{1}{\v^2}-1\right) L\bg^\v + \L \bg^\v = \frac{1}{\v}\Gamma(\bg^\v, \bg^\v) + \v\tilde{\Gamma}(\bg^\v,\bg^\v).
\end{equation}

Applying $\p_x^\a$ to \eqref{Scaled Boltzmann equation for gaseous mixtures BE equation-1} with $0\leq |\a|\leq s$, multiplying it by $\p_x^\a \bg^\v$, and taking the integration with respect to $x$ and $v$, we have
\begin{multline}\label{A pure spatial estimate-0}
\frac{1}{2}\frac{d}{dt} \|\p_x^\a \bg^\v\|_{L^2_x L^2_v}^2 + \delta \left(\frac{1}{\v^2}-1\right) \|\p_x^\a\mathbf{P}^{\perp}\bg^\v\|_{L^2_x L^2_v(\nu)}^2 + \delta\|\p_x^\a\mathbb{P}^{\perp}\bg^\v\|_{L^2_x L^2_v(\nu)}^2\\[4pt]
\leq \underbrace{\left(\frac{1}{\v}-\v\right) \l\p_x^\a\Gamma(\bg^\v,\bg^\v), \p_x^\a \bg^\v\r_{x,v}}_{A_{11}} + \underbrace{\v\l\p_x^\a\tilde{\Gamma}(\bg^\v,\bg^\v),\p_x^\a \bg^\v\r_{x,v}}_{A_{12}},
\end{multline}
where Lemma \ref{The coercivity of the linearized collision operators} is utilized.

Recalling the definition of $\Gamma(\bg^\v,\bg^\v)$ in \eqref{Linear and nonlinear opeartors} and using the decomposition \eqref{decom1}, we can split $A_{11}$ into the following four terms:
\begin{multline*}
A_{11}=\underbrace{\left(\frac{1}{\v}-\v\right) \l\p_x^\a\Gamma(\mathbf{P}^{\perp}\bg^\v,\mathbf{P}^{\perp}\bg^\v),\p_x^\a \bg^\v\r_{x,v}}_{A_{111}} + \underbrace{\left(\frac{1}{\v}-\v\right)\l\p_x^\a\Gamma(\mathbf{P}^{\perp}\bg^\v,\mathbf{P}\bg^\v),\p_x^\a \bg^\v\r_{x,v}}_{A_{112}}\\
\underbrace{\left(\frac{1}{\v}-\v\right)\l\p_x^\a\Gamma(\mathbf{P}\bg^\v,\mathbf{P}^{\perp}\bg^\v),\p_x^\a \bg^\v\r_{x,v}}_{A_{113}} + \underbrace{\left(\frac{1}{\v}-\v\right)\l\p_x^\a\Gamma(\mathbf{P}\bg^\v,\mathbf{P}\bg^\v),\p_x^\a \bg^\v \r_{x,v}}_{A_{114}}.
\end{multline*}

For $A_{111}$, by applying the Leibniz rule, we have, for $s \geq 3$,
\begin{equation*}
\begin{aligned}
|A_{111}|\lesssim & \frac{1}{\v} \big| \l \Gamma(\p_x^\a\mathbf{P}^{\perp}\bg^\v, \mathbf{P}^{\perp}\bg^\v)
+ \Gamma(\mathbf{P}^{\perp}\bg^\v, \p_x^\a\mathbf{P}^{\perp}\bg^\v), \, \p_x^\a(\mathbf{I}-\mathbf{P})\bg^\v + \p_x^\a\mathbf{P}\bg^\v\r_{x,v} \big|\\[4pt]
& +\frac{1}{\v}\sum_{1 \leq |\tilde{\a}| \leq |\a|-1} \big| \l\Gamma(\p_x^{\tilde{\a}}\mathbf{P}^{\perp}\bg^\v, \p_x^{\a-\tilde{\a}}\mathbf{P}^{\perp}\bg^\v)
+ \Gamma(\p_x^{\tilde{\a}}\mathbf{P}^{\perp}\bg^\v, \p_x^{\a-\tilde{\a}}\mathbf{P}^{\perp}\bg^\v), \p_x^\a\mathbf{P}^{\perp}\bg^\v + \p_x^\a\mathbf{P}\bg^\v \r_{x,v} \big|\\[4pt]
\lesssim & \frac{1}{\v} \Big[ \|\mathbf{P}^{\perp} \bg^\v\|_{L^\infty_x L^2_v}\|\|\p_x^\a\mathbf{P}^{\perp}\bg^\v\|_{L^2_x L^2_v(\nu)} + \|\mathbf{P}^{\perp}\bg^\v\|_{L^\infty_x L^2_v(\nu)}\| \|\p_x^\a\mathbf{P}^{\perp}\bg^\v\|_{L^2_x L^2_v} \Big] \c \\[4pt]
&\qquad \qquad \Big[ \|\p_x^\a\mathbf{P}^{\perp}\bg^\v\|_{L^2_x L^2_v(\nu)} + \|\p_x^\a\mathbf{P}\bg^\v\|_{L^2_x L^2_v(\nu)}\Big] \\[4pt]
& +\frac{1}{\v} \sum_{1\leq|\tilde{\a}|\leq|\a|-1} \Big[ \|\p_x^{\tilde{\a}}\mathbf{P}^{\perp}\bg^\v\|_{L^4_x L^2_v}\| \|\p_x^{\a-\tilde{\a}}\mathbf{P}^{\perp}\bg^\v\|_{L^4_x L^2_v(\nu)} + \|\p_x^{\tilde{\a}}\mathbf{P}^{\perp}\bg^\v\|_{L^4_x L^2_v(\nu)}\|\p_x^{\a-\tilde{\a}}\mathbf{P}^{\perp}\bg^\v\|_{L^4_x L^2_v}\Big] \c \\[4pt]
&\qquad \qquad \Big[ \|\p_x^\a\mathbf{P}^{\perp}\bg^\v\|_{L^2_x L^2_v(\nu)}+\|\p_x^\a\mathbf{P}\bg^\v\|_{L^2_x L^2_v(\nu)} \Big]\\[4pt]
\lesssim &\ \frac{1}{\v} \Big[\|\mathbf{P}^{\perp}\bg^\v\|_{H^2_x L^2_v}\|\|\p_x^\a\mathbf{P}^{\perp}\bg^\v\|_{L^2_x L^2_v(\nu)}+\|\mathbf{P}^{\perp}\bg^\v\|_{H^2_x L^2_v(\nu)}\| \|\p_x^\a\mathbf{P}^{\perp}\bg^\v\|_{L^2_x L^2_v}\Big] \c \\[4pt]
&\qquad \qquad \Big[ \|\p_x^\a\mathbf{P}^{\perp}\bg^\v\|_{L^2_x L^2_v(\nu)} + \|\p_x^\a\mathbf{P}\bg^\v\|_{L^2_x L^2_v(\nu)}\Big] \\[4pt]
&+\frac{1}{\v}\sum_{1\leq|\tilde{\a}|\leq|\a|-1} \Big[ \|\p_x^{\tilde{\a}}\mathbf{P}^{\perp}\bg^\v\|_{H^1_x L^2_v}\| \|\p_x^{\a-\tilde{\a}}\mathbf{P}^{\perp}\bg^\v\|_{H^1_x L^2_v(\nu)}+\|\p_x^{\tilde{\a}}\mathbf{P}^{\perp}\bg^\v\|_{H^1_x L^2_v(\nu)}\|\p_x^{\a-\tilde{\a}}\mathbf{P}^{\perp}\bg^\v\|_{L^4_x L^2_v}\Big] \c \\[4pt]
&\qquad \qquad \Big[ \|\p_x^\a\mathbf{P}^{\perp}\bg^\v\|_{L^2_x L^2_v(\nu)}+\|\p_x^\a\mathbf{P}\bg^\v\|_{L^2_x L^2_v(\nu)}\Big]\\[4pt]
\lesssim& \ \frac{1}{\v}\|\bg^\v\|_{H^s_x L^2_v} \|\mathbf{P}^{\perp}\bg^\v\|_{H^s_x L^2_v(\nu)}^2\\[4pt]
\lesssim& \ \mathcal{E}_s^{\frac{1}{2}}(t)\mathbb{D}_s(t),
\end{aligned}
\end{equation*}
where we apply the H$\ddot{\text{o}}$lder inequality and Lemma \ref{The estimate of the symmetric operators} in the second inequality, and the Sobolev inequality is utilized in the third inequality above.

By applying the similar argument as $A_{111}$, we also obtain that
\begin{equation*}
\begin{aligned}
|A_{112}|+|A_{113}|+|A_{113}|+|A_{114}|\lesssim\mathcal{E}_s^{\frac{1}{2}}(t)\mathbb{D}_s(t).
\end{aligned}
\end{equation*}

Therefore, we have
\begin{equation}\label{A pure estimate of Gamma}
\begin{aligned}
|A_{11}|\lesssim\mathcal{E}_s^{\frac{1}{2}}(t)\mathbb{D}_s(t).
\end{aligned}
\end{equation}

For $A_{12}$, by applying $\bg^\v=\mathbb{P}^{\perp}\bg^\v+\mathbb{P}\bg^\v$, we can decompose it as follows:
\begin{equation*}
\begin{aligned}
A_{12}=&\underbrace{\left(\frac{1}{\v}-\v\right) \l\p_x^\a\tilde{\Gamma}(\mathbb{P}^{\perp}\bg^\v,\mathbb{P}^{\perp}\bg^\v),\p_x^\a \bg^\v\r_{x,v}}_{A_{121}}
+\underbrace{\left(\frac{1}{\v}-\v\right) \l\p_x^\a\tilde{\Gamma}(\mathbb{P}^{\perp}\bg^\v,\mathbb{P}\bg^\v),\p_x^\a \bg^\v\r_{x,v}}_{A_{122}}\\
&\underbrace{\left(\frac{1}{\v}-\v\right)\l\p_x^\a\tilde{\Gamma}(\mathbb{P}\bg^\v,\mathbb{P}^{\perp} \bg^\v),\p_x^\a \bg^\v\r_{x,v}}_{A_{123}}
+\underbrace{\left(\frac{1}{\v}-\v\right)\l\p_x^\a\tilde{\Gamma}(\mathbb{P}g^\v,\mathbb{P}\bg^\v),\p_x^\a \bg^\v\r_{x,v}}_{A_{124}},
\end{aligned}
\end{equation*}
and by further following the similar estimate as $A_{111}$, we find
\begin{equation}\label{A pure estimate of Gamma-tilde}
|A_{12}| \lesssim \mathcal{E}_s^{\frac{1}{2}}(t)\mathbb{D}_s(t).
\end{equation}

Finally, the proof can be completed by combining \eqref{A pure spatial estimate-0}, \eqref{A pure estimate of Gamma} and \eqref{A pure estimate of Gamma-tilde} as well as summing up for all $0\leq |\a|\leq s$.
\end{proof}


\subsubsection{Energy estimates on macroscopic parts: fluid dissipation}
\label{subsubsec:Fluid-dissipations}

In this part, we will estimate the macroscopic part $\mathbb{P} \bg^{\v}$ in \eqref{Scaled gas mixture BE system-1} by using the well-known \textit{Macro-Micro decomposition} method.

To achieve this, we first introduce the following linearly independent basis in $L^2_v$, i.e., the so-called \textit{seventeen moments} basis in \cite{JL22}:
\begin{equation}\label{Linear basis}
\mathfrak{B} = \big\{\b^l(v),\,\b^l_i(v),\,\b_i(v),\,\tilde{\b}_i(v),\,\b_{jk}(v); \,\,1\leq i\leq 3,\,1\leq j<k\leq 3\big\}
\end{equation}
with $l \in \{1,2\}$, where
\begin{equation*}
\begin{aligned}
\b^1(v)=& [1,0]^\top \M, \quad \b^2(v)=[0,1]^\top \M,\\[4pt]
\b_i^1(v)=& [v_i,0]^\top \M, \quad \b_i^2(v)=[0,v_i]^\top \M, \quad \b_i(v)= \left[v_i^2,\,v_i^2\right]^\top \M,\\[4pt]
\tilde{\b}_i(v)=& [v_i|v|^2,v_i|v|^2]^\top \M, \quad \b_{jk}(v)=[v_jv_k,v_jv_k]^\top \M.
\end{aligned}
\end{equation*}
Then, we can define a projection $\mathcal{P}_{\mathfrak{B}}:\,L^2_v \mapsto \text{Span}\{\mathfrak{B}\}$: for any $\bff\in L^2_v$,
\begin{equation}\label{An equivalent form}
\begin{aligned}
\mathcal{P}_{\mathfrak{B}}\bff=&\sum_{l=1}^2 f^l\b^l(v)+\sum_{l=1}^2\sum_{i=1}^3 f^l_i\b^l_i(v)+\sum_{i=1}^3 f_i\b_i(v)\\
&+\sum_{i=1}^3\tilde{f}_i\tilde{\b}_i(v)+\sum_{1\leq i<j\leq 3}f_{ij}\b_{ij}(v),
\end{aligned}
\end{equation}
where the coefficients $f^l,\,f_i^l,\,f_i,\,\tilde{f}_i$ and $f_{ij}$ only depend on $\bff$.

Now, to manifest the dissipative structure of the fluid part of the $\mathbb{P} \bg^{\v}$, we are in a position to apply the \textit{Macro-Micro decomposition} to the scaled BEGM equation \eqref{Scaled gas mixture BE system-1}.

\textbf{Step I:} By substituting $\bg^\v = \mathbb{P}\bg^\v+\mathbb{P}^\perp \bg^\v$ into \eqref{Scaled gas mixture BE system-1}, we obtain
\begin{equation}\label{Decomposion of g}
\v\p_t\mathbb{P} \bg^\v+v\c\n_x\mathbb{P} \bg^\v=\Theta(\mathbb{P}^\perp \bg^\v)+H(\bg^\v),
\end{equation}
where
\begin{equation}\label{Definition of Theta and H}
\begin{aligned}
\Theta(\mathbb{P}^\perp \bg^\v):= &-\Big\{ \v\p_t\mathbb{P}^\perp \bg^\v+v\c\n_x\mathbb{P}^\perp \bg^\v+\left(\frac{1}{\v}-\v\right) L\bg^\v + \v\mathcal{L}\bg^\v \Big\},\\[4pt]
H(\bg^\v):= &\Gamma(\bg^\v,\bg^\v)+\v^2\tilde{\Gamma}(\bg^\v,\bg^\v).
\end{aligned}
\end{equation}
On the other hand, if inserting the definition of $\mathbb{P}\bg^\v$ in \eqref{Fluid part of g with L-1}, we have
\begin{equation}\label{Mathbf of Pg}
\begin{aligned}
\v\p_t\mathbb{P} \bg^\v& +v\c\n_x\mathbb{P} \bg^\v=\sum_{l=1}^2\v \p_t \left(\rho^l_\v-\frac{3}{2}\theta_\v\right) \b^l + \sum_{l=1}^2 \sum_{i=1}^3 \left[\v\p_t u_{\v i}+\p_{x_i} \left(\rho_{\v}^l-\frac{3}{2}\theta_\v\right)\right] \b^l_i \\
& + \sum_{i=1}^3 \left(\frac{1}{2}\v\p_t\theta_\v+\p_{x_i}u_{\v i}\right)\b_i + \sum_{i=1}^3\frac{1}{2}\p_{x_i}\theta_\v\tilde{\b}_i+\sum_{1\leq i<j\leq 3}(\p_{x_i}u_{\v j}+\p_{x_j}u_{\v i})\b_{ij},
\end{aligned}
\end{equation}
which belongs to the space $\text{Span}\{\mathfrak{B}\}$.

Then, by projecting the equation \eqref{Decomposion of g}
into $\text{Span}\{\mathfrak{B}\}$ and utilizing \eqref{Mathbf of Pg}, we find
\begin{equation}\label{First method of micro-macro decomposition}
\begin{aligned}
&\b^l(v):\, \v\p_t(\rho^l_\v-\frac{3}{2}\theta_\v)=\Theta^l+H^l,\quad1\leq l\leq 2,\\[4pt]
&\b^l_i(v):\, \v\p_t u_{\v i}+\p_{x_i}(\rho_{\v}^l-\frac{3}{2}\theta_\v)=\Theta_i^l+H^l_i,\quad 1\leq i\leq 3,\,1\leq l\leq 2,\\[4pt]
&\b_i(v):\, \frac{1}{2}\v\p_t \theta_{\v}+\p_{x_i}u_{\v i}=\Theta_i+H_i,\quad 1\leq i\leq 3,\\[4pt]
&\tilde{\b}_i(v):\, \frac{1}{2}\p_{x_i} \theta_{\v}=\tilde{\Theta}_i+\tilde{H}_i,\quad 1\leq i\leq 3,\\[4pt]
&\b_{ij}(v):\, \p_{x_i}u_{\v j}+\p_{x_j}u_{\v i}=\Theta_{ij}+H_{ij},\quad 1\leq i<j\leq 3,
\end{aligned}
\end{equation}
where all the symbols $\Theta$ and $H$ with various indexes are the coefficients of $\mathcal{P}_{\mathfrak{B}}\Theta(\mathbf{P}^\perp \bg^\v)$ and $\mathcal{P}_{\mathfrak{B}}H(\bg^\v)$, respectively.

\textbf{Step II:}
Since $\phi_i(v)\in\up{Ker}\L$ for $1\leq i\leq 6$,
we can project $\bg^\v$ in \eqref{Scaled gas mixture BE system-1} into $\text{Ker}\mathcal{L}$ by multiplying both sides of \eqref{Scaled gas mixture BE system-1} with
$\phi_1(v),\, \phi_2(v),\, \frac{1}{2}\phi_3(v),\, \frac{1}{2}\phi_4(v),\, \frac{1}{2}\phi_5(v),\, \frac{2}{3}\phi_6(v)$
and integrating over $v\in\R^3$,
then a direct calculation shows that
\begin{equation}\label{Seond method of micro-macro decomposition}
\left\{
\begin{aligned}
&\v\p_t\rho_\v^1+\up{div}_x u_\v=\l\Gamma(\bg^\v,\bg^\v)+\v^2\tilde{\Gamma}(\bg^\v,\bg^\v)-v\c\n_x\mathbb{P}^\perp \bg^\v,\phi_1(v)\r_v,\\[4pt]
&\v\p_t\rho_\v^2+\up{div}_x u_\v=\l\Gamma(\bg^\v,\bg^\v)+\v^2\tilde{\Gamma}(\bg^\v,\bg^\v)-v\c\n_x\mathbb{P}^\perp \bg^\v,\phi_2(v)\r_v,\\[4pt]
&\v\p_t u_{\v i}+\p_{x_i} \left(\frac{\rho_\v^1+\rho_\v^2}{2}+\theta_\v\right)=\frac{1}{2}\l\Gamma(\bg^\v,\bg^\v)+\v^2\tilde{\Gamma}(\bg^\v,\bg^\v)-v\c\n_x\mathbb{P}^\perp \bg^\v,\phi_{i+2}(v)\r_v,\,\text{for}\,1\leq i\leq 3,\\[4pt]
&\v\p_t\theta_\v+\frac{1}{3}\up{div}_x u_\v=\frac{2}{3}\l\Gamma(\bg^\v,\bg^\v)+\v^2\tilde{\Gamma}(\bg^\v,\bg^\v)-v\c\n_x\mathbb{P}^\perp \bg^\v,\phi_6(v)\r_v.
\end{aligned}
\right.
\end{equation}

Based on \eqref{First method of micro-macro decomposition} and \eqref{Seond method of micro-macro decomposition}, we can find the uniform energy estimate concerning the fluid part $\mathbb{P}\bg^\v$.

\begin{lemma}\label{fluid-part-energy}
Under the assumptions of Theorem \ref{Global-in-time solution of BE}, let $\bg^\v$ be the solution to the system \eqref{Scaled gas mixture BE system-1}-\eqref{Initial data}, then, for $0 < \v < 1$, there exist $c,\, C_2 > 0$ independent of $\v$ such that
\begin{equation}\label{The dissipation of the fluid part}
\begin{aligned}
\|\n_x\mathbb{P}\bg^\v\|_{H^{s-1}_xL^2_v}^2 + \v c\frac{d}{dt}E^{int}_s(t) \leq C_2 \left(\frac{1}{\v^2}\|\mathbf{P}^\perp \bg^\v\|_{H^s_x L^2_v(\nu)}^2 + \|\mathbb{P}\bg^\v\|_{H^s_x L^2_v(\nu)}^2 + \E_s^{\frac{1}{2}}(t)\mathbb{D}_s(t) \right),
\end{aligned}
\end{equation}
where the quantity $E^{int}_s(t)$ is defined as
\begin{equation}\label{Temporary energy}
\begin{aligned}
E^{int}_s(t)=& \sum_{0\leq |\a|\leq s-1}\Big\{\sum_{i,j=1}^3 72\l\p_{x_j}\p_x^\a\P^\perp \bg^\v,\p_x^\a u_{\v i}\zeta_{ij}\r_{x,v}+\sum_{i=1}^3 12\l\p_{x_i}\p_x^\a\P^\perp \bg^\v,\p_x^\a \theta_{\v}\zeta_{i}\r_{x,v}\\[4pt]
& + \sum_{i=1}^3\l\p_x^\a u_{\v i},\p_{x_i}\p_x^\a\rho_\v^1+\p_{x_i}\p_x^\a\rho_\v^2\r_x + \l\p_x^\a\P^\perp \bg^\v,\p_{x_i}\p_x^\a\rho_\v^1\zeta_i^1 + \p_{x_i}\p_x^\a\rho_\v^2\zeta_i^2\r_{x,v}\Big\}.
\end{aligned}
\end{equation}
and $\zeta^1_i,\,\zeta^2_i,\,\zeta_{i},\,\tilde{\zeta}_i$ and $\zeta_{ij}$ are some fixed linear combinations of the basis functions of $\text{Span}\{\mathfrak{B}\}$.
\end{lemma}

\begin{proof}
The proof will be divided into the following four steps.

\textbf{Step 1: Estimate of $\|\n_x u_\v\|_{H^{s-1}_x}^2$ for $s\geq 3$.}

For any $0\leq|\a|\leq s-1$, by using the last $u$-equation and the third $\theta$-equation in \eqref{First method of micro-macro decomposition}, we have
\begin{equation*}\label{Estimate of u-0}
\begin{aligned}
-\Delta_x\p_x^\a u_{\v i}=&-\sum_{j=1}^3\p_{x_j}\p_{x_j}\p_x^\a u_{\v i}\\
=&-\sum_{j\neq i}\p_{x_j}\p_{x_j}\p_x^\a u_{\v i}-\p_{x_i}\p_{x_i}\p_x^\a u_{\v i}\\
=&-\sum_{j\neq i}\p_x^\a\p_{x_j}(-\p_{x_i}u_{\v j}+\Theta_{ij}+H_{ij})-\p_{x_i}\p_x^\a \left(-\frac{1}{2}\v\p_t\theta_\v+\Theta_i+H_i \right)\\
=&-\sum_{j\neq i}\p_x^\a\p_{x_i} \left(-\frac{1}{2}\v\p_t \theta_{\v}+\Theta_j+H_j\right) - \sum_{j\neq i}\p_{x_j}\p_x^\a (\Theta_{ij}+H_{ij}) \\
&-\p_{x_i}\p_x^\a \left( -\frac{1}{2}\v\p_t\theta_\v+\Theta_i+H_i \right)\\
=&-\frac{1}{2}\v\p_t\p_{x_i}\p_x^\a\theta_\v + \sum_{\Lambda\in\{\Theta,H\}} \Big[ \sum_{j\neq i}(\p_{x_i}\p_x^\a\Lambda_j-\p_{x_j}\p_x^\a\Lambda_{ij})-\p_{x_i}\p_x^\a\Lambda_i \Big].
\end{aligned}
\end{equation*}
It follows $\mathcal{P}_{\mathfrak{B}}$ in \eqref{An equivalent form} that there is a certain linear combinations $\zeta_{ij} \in \text{Span}\{\mathfrak{B}\}$ such that
\begin{equation*}
\begin{aligned}
\sum_{\Lambda\in\{\Theta,H\}}\Big[\sum_{j\neq i}(\p_{x_i}\p_x^\a\Lambda_j-\p_{x_j}\p_x^\a\Lambda_{ij})-\p_{x_i}\p_x^\a\Lambda_i\Big]=\sum_{j=1}^3\p_{x_j}\p_x^\a\l\Theta(\mathbb{P}^\perp \bg^\v)+H(g^\v),\zeta_{ij}(v)\r_v.
\end{aligned}
\end{equation*}
Then, we deduce that
\begin{equation}\label{Estimate of u-1}
-\Delta_x\p_x^\a u_{\v i}=-\frac{1}{2}\v\p_t\p_{x_i}\p_x^\a\theta_\v+\sum_{j=1}^3\p_{x_j}\p_x^\a\l\Theta(\mathbb{P}^\perp \bg^\v)+H(\bg^\v),\zeta_{ij}(v)\r_v.
\end{equation}
Therefore, combining with $\eqref{Seond method of micro-macro decomposition}_4$ and \eqref{Estimate of u-1}, we obtain
\begin{equation}\label{Estimate of u-2}
\begin{aligned}
-\Delta_x\p_x^\a u_{\v i}-\frac{1}{6}\p_{x_i}\p_x^\a\up{div}_x u_\v=&-\frac{1}{3}\p_{x_i}\p_x^\a\l\Gamma(\bg^\v,\bg^\v)+\v^2\tilde{\Gamma}(\bg^\v,\bg^\v)-v\c\n_x\mathbb{P}^\perp \bg^\v, \,\phi_6\r_v\\
&+\sum_{j=1}^3\p_{x_j}\p_x^\a\l\Theta(\mathbb{P}^\perp \bg^\v)+H(\bg^\v), \, \zeta_{ij}\r_v.
\end{aligned}
\end{equation}

Furthermore, multiplying by $\p_x^\a u_{\v i}$ to both-hand-sides of \eqref{Estimate of u-2} with $0\leq |\a|\leq s-1$, integrating over $x\in\R^3$ and summing up $1\leq i\leq 3$, we have
\begin{equation}\label{Estimate of u-3}
\begin{aligned}
&\|\n_x\p_x^\a u_\v\|_{L^2_x}^2+\frac{1}{6}\|\p_x^\a\up{div}_x u_\v\|_{L^2_x}^2\\[4pt]
=&\underbrace{\frac{1}{3}\sum_{i=1}^3\l\p_x^\a\l\Gamma(\bg^\v,\bg^\v)+\v^2\tilde{\Gamma}(\bg^\v,\bg^\v)-v\c\n_x\mathbb{P}^\perp \bg^\v, \,\phi_6 \r_v, \, \p_{x_i}\p_x^\a u_{\v i}\r_x}_{B_1}\\
&+\underbrace{\sum_{i=1}^3\sum_{j=1}^3\l\p_{x_j}\p_x^\a\l\Theta(\mathbb{P}^\perp g^\v)+H(g^\v), \, \zeta_{ij}\r_v, \, \p_x^\a u_{\v i}\r_x}_{B_2}.
\end{aligned}
\end{equation}

For $B_1$, by substituting decomposition \eqref{decom1} into $\Gamma(\bg^\v,\bg^\v)$ and \eqref{decom2} into $\tilde{\Gamma}(\bg^\v,\bg^\v)$, and also using the similar argument as $A_{111}$ in \eqref{A pure estimate of Gamma}, we have
\begin{equation}\label{Estimate of B1}
\begin{aligned}
|B_1|\leq& C\|\mathbb{P}^\perp \bg^\v\|_{H^s_x L^2_v(\nu)}\|\p_x^\a\up{div}_x u_\v\|_{L^2_x}+C\E^{\frac{1}{2}}(t)\mathbb{D}(t)\\[4pt]
\leq& \frac{1}{24}\|\p_x^\a\up{div}_x u_\v\|_{L^2_x}^2+C\|\mathbb{P}^\perp \bg^\v\|_{H^s_x L^2_v(\nu)}^2+C\E^{\frac{1}{2}}(t)\mathbb{D}(t).
\end{aligned}
\end{equation}

For $B_{2}$, it can be divided into five parts as follows:
\begin{equation}\label{Estimate of B2-0}
\begin{aligned}
B_2=&\underbrace{-\sum_{i,j=1}^3\l\p_{x_j}\p_x^\a\l\v\p_t\P^\perp \bg^\v, \, \zeta_{ij}\r_v,\, \p_x^\a u_{\v i}\r_x}_{B_{21}}\underbrace{-\sum_{i,j=1}^3\l\p_{x_j}\p_x^\a\l v\c\n_x\P^\perp \bg^\v, \, \zeta_{ij} \r_v, \,\p_x^\a u_{\v i}\r_x}_{B_{22}}\\
&\underbrace{-\left(\frac{1}{\v}-\v\right) \sum_{i,j=1}^3\l\p_{x_j}\p_x^\a\l L(\mathbf{P}^\perp \bg^\v),\, \zeta_{ij} \r_v, \, \p_x^\a u_{\v i}\r_x}_{B_{23}}
\underbrace{-\v\sum_{i,j=1}^3\l\p_{x_j}\p_x^\a\l \L(\mathbb{P}^\perp \bg^\v),\,\zeta_{ij} \r_v, \, \p_x^\a u_{\v i}\r_x}_{B_{24}}\\
&\underbrace{\sum_{i,j=1}^3\l\p_{x_j}\p_x^\a\l\Gamma(\bg^\v,\bg^\v) + \v^2\tilde{\Gamma}(\bg^\v,\bg^\v),\,\zeta_{ij}\r_v, \, \p_x^\a u_{\v i}\r_x}_{B_{25}}.
\end{aligned}
\end{equation}
For $B_{21}$, recalling the $u_\v$-equation in $\eqref{Seond method of micro-macro decomposition}_3$, we have
\begin{equation}\label{Estimate of B21}
\begin{aligned}
B_{21}=&-\v\frac{d}{dt}\sum_{i,j=1}^3\l\p_{x_j}\p_x^\a\P^\perp \bg^\v, \, \p_x^\a u_{\v i}\zeta_{ij}\r_{x,v}
-\sum_{i,j=1}^3\l\p_{x_j}\p_x^\a\P^\perp \bg^\v,\, \p_x^\a \left(\frac{\rho_\v^1+\rho_\v^2}{2}+\theta_\v\right)\r_{x,v}\\
&+\frac{1}{2}\sum_{i,j=1}^3\l\p_{x_j}\p_x^\a\P^\perp \bg^\v, \, \p_x^\a\l\Gamma(\bg^\v,\bg^\v) + \v^2\tilde{\Gamma}(\bg^\v,\bg^\v)-v\c\n_x\P^\perp \bg^\v, \, \phi_{i+2} \r_v \,\zeta_{ij} \r_{x,v}\\
\leq& -\v\frac{d}{dt}\sum_{i,j=1}^3\l\p_{x_j}\p_x^\a\P^\perp \bg^\v, \, \p_x^\a u_{\v i}\zeta_{ij}\r_{x,v} + \delta_1 \left( \|\n_x\p_x^\a\rho_\v^1\|_{L^2_x}^2+\|\n_x\p_x^\a\rho_\v^2\|_{L^2_x}^2+\|\n_x\p_x^\a\theta_\v\|_{L^2_x}^2 \right)\\
&+C\E_s^{\frac{1}{2}}(t)\mathbb{D}(t),
\end{aligned}
\end{equation}
where the constant $\delta_1>0$ is to be determined.\\
For $B_{22}$, we have
\begin{equation}\label{Estimate of B22}
\begin{aligned}
|B_{22}|\leq&\sum_{i,j=1}^3\|\p_x^\a\P^\perp \bg^\v\|_{L^2_x L^2_v}\|\p_{x_j}\p_x^\a u_{\v i}\|_{L^2_x}\leq\frac{1}{48}\|\n_x\p_x^\a u_\v\|_{L^2_x}^2+C\|\P^\perp \bg^\v\|_{H^s_x L^2_v(\nu)}^2.
\end{aligned}
\end{equation}
where the H$\ddot{\text{o}}$lder inequality and Young inequality are used.\\
For $B_{23}$ and $B_{24}$, by applying the self-adjointness of linearized operators $L$ and $\L$, we find, for $0<\v\leq 1$,
\begin{equation}\label{Estimate of B23-B24}
\begin{aligned}
|B_{23}|+|B_{24}|=&\frac{1}{\v}\sum_{i,j=1}^3\Big|\l\p_x^\a\l\mathbf{P}^\perp \bg^\v,L(\zeta_{ij}(v))\r_v,\p_{x_j}\p_x^\a u_{\v i}\r_x\Big|\\
&+\v\sum_{i,j=1}^3\Big|\l\p_x^\a\l\P^\perp \bg^\v,\L(\zeta_{ij}(v))\r_v,\p_{x_j}\p_x^\a u_{\v i}\r_x\Big|\\
\leq&\frac{1}{48}\|\n_x\p_x^\a u_\v\|_{L^2_x}^2+\frac{C}{\v^2}\|\mathbf{P}^\perp \bg^\v\|_{H^s_x L^2_v(\nu)}^2+C\|\P^\perp \bg^\v\|_{H^s_x L^2_v(\nu)}^2.
\end{aligned}
\end{equation}
For $B_{25}$, by applying the similar argument as $A_{111}$ in \eqref{A pure estimate of Gamma}, it can be bounded by
\begin{equation}\label{Estimate of B25}
\begin{aligned}
|B_{25}|\lesssim\E_s^{\frac{1}{2}}(t)\mathbb{D}(t).
\end{aligned}
\end{equation}
Therefore, combining with \eqref{Estimate of B2-0}, \eqref{Estimate of B21}, \eqref{Estimate of B22}, \eqref{Estimate of B23-B24}, and \eqref{Estimate of B25}, we have
\begin{equation}\label{Estimate of B2}
\begin{aligned}
&B_2+\v\frac{d}{dt}\sum_{i,j=1}^3\l\p_{x_j}\p_x^\a\P^\perp \bg^\v,\p_x^\a u_{\v i}\zeta_{ij}(v)\r_{x,v} \\
\leq & \frac{1}{24}\big(\|\n_x\p_x^\a u_\v\|_{L^2_x}^2+\|\p_x^\a\up{div}_x u_\v\|_{L^2_x}^2\big)+C\E_s^{\frac{1}{2}}(t)\mathbb{D}(t)\\[4pt]
& +\delta_1\Big(\|\n_x\p_x^\a\rho_\v^1\|_{L^2_x}^2+\|\n_x\p_x^\a\rho_\v^2\|_{L^2_x}^2+\|\n_x\p_x^\a\theta_\v\|_{L^2_x}^2\Big)+\frac{C}{\v^2}\|\mathbf{P}^\perp \bg^\v\|_{H^s_x L^2_v(\nu)}^2\\[4pt]
& +C\|\P^\perp \bg^\v\|_{H^s_x L^2_v(\nu)}^2.
\end{aligned}
\end{equation}

Finally, by substituting the estimates \eqref{Estimate of B1} and \eqref{Estimate of B2} into \eqref{Estimate of u-3}, we obtain
\begin{equation}\label{Estimate of u}
\begin{aligned}
&\frac{11}{12}\|\n_x\p_x^\a u_\v\|_{L^2_x}^2+\frac{1}{12}\|\p_x^\a\up{div}_x u_\v\|_{L^2_x}^2+\v\frac{d}{dt} \sum_{i,j=1}^3\l\p_{x_j}\p_x^\a\P^\perp \bg^\v, \, \p_x^\a u_{\v i}\zeta_{ij} \r_{x,v}\\[4pt]
\leq & C\E_s^{\frac{1}{2}}(t)\mathbb{D}(t) + \delta_1\Big(\|\n_x\p_x^\a\rho_\v^1\|_{L^2_x}^2 + \|\n_x\p_x^\a\rho_\v^2\|_{L^2_x}^2+\|\n_x\p_x^\a\theta_\v\|_{L^2_x}^2\Big)\\[4pt]
& +\frac{C}{\v^2}\|\mathbf{P}^\perp \bg^\v\|_{H^s_x L^2_v(\nu)}^2+C\|\P^\perp \bg^\v\|_{H^s_x L^2_v(\nu)}^2.
\end{aligned}
\end{equation}

\textbf{Step 2: Estimate of $\|\n_x \theta_\v\|_{H^{s-1}_x}^2$ for $s\geq 3$.}

For $0\leq|\a|\leq s-1$, by applying the fourth $\theta$-equation in \eqref{First method of micro-macro decomposition} that
\begin{equation}\label{Estimate of theta-0}
\begin{aligned}
-\Delta_x\p_x^\a \theta_{\v}=-\sum_{i=1}^3\p_{x_i}\p_{x_i}\p_x^\a \theta_{\v}=&-2\sum_{i=1}^3\p_{x_i}\p_x^\a(\tilde{\Theta}_i+\tilde{H}_i)\\
=&\sum_{i=1}^3\p_{x_i}\p_x^\a\l\Theta(\P^\perp \bg^\v)+H(\bg^\v), \, \tilde{\zeta}_i\r_v,
\end{aligned}
\end{equation}
where $\tilde{\zeta}_i \in \text{Span}\{\mathfrak{B}\}$.

Multiplying both sides of \eqref{Estimate of theta-0} by $\p_x^\a \theta_{\v}$ with $0\leq |\a|\leq s-1$ and integrating over $x\in\R^3$, we have
\begin{equation}\label{Estimate of theta-1}
\begin{aligned}
\|\n_x\p_x^\a \theta_\v\|_{L^2_x}^2=&-\sum_{i=1}^3\l\p_x^\a\l\Theta(\P^\perp \bg^\v),\tilde{\zeta}_i\r_v, \, \p_{x_i}\p_x^\a \theta_{\v}\r_x
-\sum_{i=1}^3\l\p_{x_i}\p_x^\a\l H(\bg^\v),\tilde{\zeta}_{i}\r_v, \, \p_{x_i}\p_x^\a \theta_{\v}\r_x\\
=&\underbrace{\sum_{i=1}^3\l\p_x^\a\p_t\P^\perp \bg^\v, \, \tilde{\zeta}_i\p_{x_i}\p_x^\a \theta_{\v}\r_{x,v}}_{B_3}
+\underbrace{\sum_{i=1}^3\l\p_x^\a v\c\n_x\P^\perp \bg^\v, \, \tilde{\zeta}_i \p_{x_i}\p_x^\a \theta_{\v}\r_{x,v}}_{B_4}\\
&+\underbrace{\left(\frac{1}{\v}-\v\right) \sum_{i=1}^3\l\p_x^\a L(\mathbf{P}^\perp \bg^\v), \, \tilde{\zeta}_i \p_{x_i}\p_x^\a \theta_{\v}\r_{x,v}}_{B_5}
+\underbrace{\v\sum_{i=1}^3\l\p_x^\a \L(\P^\perp \bg^\v), \, \tilde{\zeta}_i \p_{x_i}\p_x^\a \theta_{\v}\r_{x,v}}_{B_6}\\
&\underbrace{-\sum_{i=1}^3\l\p_{x_i}\p_x^\a\l\Gamma(\bg^\v,\bg^\v)+\v^2\tilde{\Gamma}(\bg^\v,\bg^\v),\tilde{\zeta}_{i}\r_v, \, \p_{x_i}\p_x^\a \theta_{\v}\r_x}_{B_7}.
\end{aligned}
\end{equation}
For $B_3$, recalling the $\theta_\v$-equation in $\eqref{Seond method of micro-macro decomposition}_4$,
\begin{equation}\label{Estimate of B3}
\begin{aligned}
B_3=&-\v\frac{d}{dt}\sum_{i=1}^3\l\p_{x_i}\p_x^\a\P^\perp \bg^\v, \, \p_x^\a\theta_\v\tilde{\zeta}_i \r_{x,v} - \frac{1}{3}\sum_{i=1}^3\l\p_{x_i}\p_x^\a\P^\perp \bg^\v, \, \p_x^\a\up{div}_x u_\v\tilde{\zeta}_i \r_{x,v}\\
&+\frac{2}{3}\sum_{i=1}^3\l\p_{x_i}\p_x^\a\P^\perp \bg^\v,\p_x^\a\l\Gamma(\bg^\v,\bg^\v) +\v^2\tilde{\Gamma}(\bg^\v,\bg^\v)-v\c\n_x\P^\perp \bg^\v, \, \phi_6 \r_v \,\tilde{\zeta}_i \r_{x,v}\\
\leq &-\v\frac{d}{dt}\sum_{i=1}^3\l\p_{x_i}\p_x^\a\P^\perp \bg^\v,\p_x^\a\theta_\v\tilde{\zeta}_i \r_{x,v} + \delta_1\|\p_x^\a\up{div}_x u_\v\|_{L^2_x}^2 + C\|\P^\perp \bg^\v\|_{H^s_x L^2_v(\nu)}^2\\
& + C\E_s^{\frac{1}{2}}(t)\mathbb{D}(t).
\end{aligned}
\end{equation}
For $B_4$,
\begin{equation}\label{Estimate of B4}
\begin{aligned}
|B_4|\leq& \sum_{i=1}^3\|\n_x\p_x^\a\P^\perp \bg^\v\|_{L^2_x L^2_v}\|\tilde{\phi}_i(v)\|_{L^2_v}\|\n_{x_i}\p_x^\a\theta_\v\|_{L^2_x}\\
\leq& \frac{1}{24} \|\n_{x}\p_x^\a\theta_\v\|^2_{L^2_x} +C\|\P^\perp \bg^\v\|_{H^s_x L^2_v(\nu)}^2.
\end{aligned}
\end{equation}
For $B_5$ and $B_6$, by using the similar argument as $B_{23},\,B_{24}$,
\begin{equation}\label{Estimate of B5-B6}
\begin{aligned}
|B_5|+|B_6|\leq & \frac{1}{\v}\sum_{i=1}^3\|\p_x^\a\mathbf{P}^\perp \bg^\v\|_{L^2_x L^2_v}\| L\tilde{\zeta}_i\|_{L^2_v}\|\p_{x_i}\p_x^\a\theta_\v\|_{L^2_x}\\
&+\v\sum_{i=1}^3\|\p_x^\a\P^\perp \bg^\v\|_{L^2_x L^2_v}\| \L\tilde{\zeta}_i\|_{L^2_v}\|\p_{x_i}\p_x^\a\theta_\v\|_{L^2_x}\\
\leq & \frac{1}{24} \|\n_{x}\p_x^\a\theta_\v\|^2_{L^2_x}+\frac{C}{\v^2}\|\mathbf{P}^\perp \bg^\v\|_{H^s_x L^2_v(\nu)}^2+C\v^2\|\P^\perp \bg^\v\|_{H^s_x L^2_v(\nu)}^2.
\end{aligned}
\end{equation}
For $B_7$, by using the similar argument as $B_{25}$,
\begin{equation}\label{Estimate of B7}
\begin{aligned}
|B_7|\lesssim&\E_s^{\frac{1}{2}}(t)\mathbb{D}(t).
\end{aligned}
\end{equation}

Hence, by inserting the estimates \eqref{Estimate of B3}, \eqref{Estimate of B4},\eqref{Estimate of B5-B6} and \eqref{Estimate of B7} into \eqref{Estimate of theta-1}, we have
\begin{equation}\label{Estimate of theta}
\begin{aligned}
&\frac{11}{12}\|\n_x\p_x^\a\theta_\v\|_{L^2_x}^2+\v\frac{d}{dt}\sum_{i=1}^3\l\p_{x_i}\p_x^\a\P^\perp \bg^\v, \, \p_x^\a\theta_\v\tilde{\zeta}_i \r_{x,v}\\
\leq &\delta_1\|\p_x^\a\up{div}_x u_\v\|_{L^2_x}^2 +\frac{C}{\v^2}\|\mathbf{P}^\perp \bg^\v\|_{H^s_x L^2_v(\nu)}^2+C\|\P^\perp \bg^\v\|_{H^s_x L^2_v(\nu)}^2+C\E_s^{\frac{1}{2}}(t)\mathbb{D}(t),
\end{aligned}
\end{equation}
where the constant $\delta_1>0$ is to be determined.

\textbf{Step 3: Estimate of $\|\n_x \rho^1_\v\|_{H^{s-1}_x}^2$ and $\|\n_x \rho^2_\v\|_{H^{s-1}_x}^2$ for $s\geq 3$.}

For $0\leq|\a|\leq s-1$, by using the second $\rho^l$-equation in $\eqref{First method of micro-macro decomposition}_2$ for $l \in \{1,2\}$, we have
\begin{equation}\label{Estimate of rho-0}
\begin{aligned}
-\Delta_x\p_x^\a \rho^l_{\v}=&-\sum_{i=1}^3\p_{x_i}\p_{x_i}\p_x^\a \rho^l_{\v}\\
=&\sum_{i=1}^3\p_{x_i}\p_x^\a \left(\v\p_t u_{\v i}-\frac{3}{2}\p_{x_i}\theta_\v-\Theta_i^l-H_i^l\right)\\
=&\sum_{i=1}^3\p_{x_i}\p_x^\a\v\p_t u_{\v i}-\frac{3}{2}\sum_{i=1}^3\p_{x_i}\p_{x_i}\p_x^\a\theta_\v + \sum_{i=1}^3\p_{x_i} \p_x^\a \l\Theta(\P^\perp \bg^\v)+H(\bg^\v), \, \zeta_i^l \r_v
\end{aligned}
\end{equation}
where $\zeta_i^l \in \text{Span}\{\mathfrak{B}\}$.
Then, by multiplying both sides of  \eqref{Estimate of rho-0} by $\p_x^\a \rho^l_{\v}$ with $0\leq |\a|\leq s-1$, integrating over $x\in\R^3$ and summing up $i=1,2$, we find
\begin{equation}\label{Estimate of rho-1}
\begin{aligned}
&\|\n_x\p_x^\a \rho^1_\v\|_{L^2_x}^2+\|\n_x\p_x^\a \rho^1_\v\|_{L^2_x}^2=-\v\frac{d}{dt} \sum_{i=1}^3\l\p_x^\a u_{\v i}, \, \p_{x_i}\p_x^\a\rho_\v^1+\p_{x_i}\p_x^\a\rho_\v^2\r_{x}\\
&+\underbrace{\sum_{i=1}^3\l\p_x^\a u_{\v i},\, \v\p_t\p_{x_i}\p_x^\a\rho_\v^1+\v\p_t\p_{x_i}\p_x^\a\rho_\v^2\r_{x}}_{B_{81}}
+\underbrace{\frac{3}{2}\sum_{i=1}^3\l\p_{x_i}\p_x^\a\theta_\v, \, \p_{x_i}\p_x^\a\rho_\v^1+\p_{x_i}\p_x^\a\rho_\v^2\r_{x}}_{B_{82}}\\
&+\underbrace{\sum_{i=1}^3\l\p_x^\a\Theta(\P^\perp \bg^\v),\, \zeta_i^1 \p_{x_i}\p_x^\a\rho_\v^1 + \zeta_i^2 \p_{x_i}\p_x^\a\rho_\v^2 \r_{x,v}}_{B_{83}}\\
&+\underbrace{\sum_{i=1}^3\l\p_x^\a H(\bg^\v),\, \zeta_i^1\p_{x_i}\p_x^\a\rho_\v^1 + \zeta_i^2 \p_{x_i}\p_x^\a\rho_\v^2 \r_{x,v}}_{B_{84}}.
\end{aligned}
\end{equation}
For $B_{81}$, recalling the $\rho_\v^l$-equation in \eqref{Seond method of micro-macro decomposition},
\begin{equation}\label{Estimate of B81}
\begin{aligned}
B_{81}=&-\sum_{i=1}^3\l\p_{x_i}\p_x^\a u_{\v i}, \, \p_x^\a\big\{\Gamma(\bg^\v,\bg^\v)+\v^2\tilde{\Gamma}(\bg^\v,\bg^\v)-v\c\n_x\P^\perp \bg^\v\big\}(\phi_1+\phi_2)\r_{x,v}\\
&+2\sum_{i=1}\l\p_{x_i}\p_x^\a u_{\v i}, \, \p_x^\a\up{div}u_\v\r_x\\
\leq&\frac{5}{2}\|\p_x^\a\up{div}_x u_\v\|_{L^2_x}^2+C\|\P^\perp \bg^\v\|_{H^s_x L^2_v(\nu)}^2+C\E_s^{\frac{1}{2}}(t)\mathbb{D}_s(t).
\end{aligned}
\end{equation}
For $B_{82}$, by applying Young inequality,
\begin{equation}\label{Estimate of B82}
\begin{aligned}
|B_{82}|=&\frac{3}{2}\sum_{i=1}^3\|\p_{x_i}\p_x^\a\theta_\v\|_{L^2_x} \left(\|\p_{x_i}\p_x^\a\rho^1_\v\|_{L^2_x}+\|\p_{x_i}\p_x^\a\rho^2_\v\|_{L^2_x}\right)\\
\leq& \frac{1}{8}\left(\|\n_{x}\p_x^\a\rho^1_\v\|^2_{L^2_x}+\|\n_{x}\p_x^\a\rho^2_\v\|^2_{L^2_x}\right) + \frac{9}{2}\|\n_{x}\p_x^\a\theta_\v\|_{L^2_x}^2.
\end{aligned}
\end{equation}
For $B_{83}$, recalling \eqref{Definition of Theta and H}, we have
\begin{equation}\label{Estimate of B83-0}
\begin{aligned}
B_{83}=&-\sum_{i=1}^3\l\p_x^\a\big\{\v\p_t\P^\perp \bg^\v+v\c\n_x\P^\perp \bg^\v+\left(\frac{1}{\v}-\v\right) L \bg^\v+\v\L \bg^\v\big\},\, \p_{x_i}\p_x^\a\rho_\v^1\zeta_i^1 +\p_{x_i}\p_x^\a\rho_\v^2\zeta_i^2 \r_{x,v}\\
\leq & -\v\frac{d}{dt}\sum_{i=1}^3\l\p_x^\a\p_t\P^\perp \bg^\v,\p_{x_i}\p_x^\a\rho_\v^1\zeta_i^1 +\p_{x_i}\p_x^\a\rho_\v^2\zeta_i^2 \r_{x,v}+\frac{1}{16} \left(\|\n_{x}\p_x^\a\rho^1_\v\|^2_{L^2_x}+\|\n_{x}\p_x^\a\rho^2_\v\|^2_{L^2_x}\right)\\
&+\frac{C}{\v^2}\|\mathbf{P}^\perp \bg^\v\|_{H^s_x L^2_v(\nu)}^2
+C\|\P^\perp \bg^\v\|_{H^s_x L^2_v(\nu)}^2
+\underbrace{\l\p_x^\a\P^\perp\bg^\v,\v\p_t\p_{x_i}\p_x^\a\rho_\v^1\zeta_i^1 +\v\p_t\p_{x_i}\p_x^\a\rho_\v^2\zeta_i^2 \r_{x,v}}_{B_{831}}.
\end{aligned}
\end{equation}
where, by recalling the $\rho_\v^l$-equation in \eqref{Seond method of micro-macro decomposition} and applying the similar argument as $B_{81}$, $B_{831}$ is bounded by
\begin{equation}\label{Estimate of B831}
\begin{aligned}
B_{831}=&-\sum_{i=1}^3\l\p_{x_i}\p_x^\a \P^\perp \bg^\v,\p_x^\a\big\{\Gamma(\bg^\v,\bg^\v)+\v^2\tilde{\Gamma}(\bg^\v,\bg^\v)-v\c\n_x\P^\perp \bg^\v\big\}(\phi_1+\phi_2)\r_{x,v}\\
&+2\sum_{i=1}\l\p_{x_i}\p_{x_i}\p_x^\a\P^\perp \bg^\v,\p_x^\a\up{div}u_\v\r_x\\
\leq&\frac{5}{2}\|\p_x^\a\up{div}_x u_\v\|_{L^2_x}^2+C\|\P^\perp \bg^\v\|_{H^s_x L^2_v(\nu)}^2+C\E_s^{\frac{1}{2}}(t)\mathbb{D}_s(t).
\end{aligned}
\end{equation}
For $B_{84}$,
\begin{equation}\label{Estimate of B84}
\begin{aligned}
B_{84}=&\sum_{i=1}^3\l\p_x^\a\big\{\Gamma(\bg^\v,\bg^\v)+\v^2\tilde{\Gamma}(\bg^\v,\bg^\v)-v\c\n_x\P^\perp \bg^\v\big\},\, \p_{x_i}\p_x^\a\rho_\v^1\zeta_i^1+\p_{x_i}\p_x^\a\rho_\v^2\zeta_i^2\r_{x,v}\\
\leq&\frac{1}{16}\left(\|\n_{x}\p_x^\a\rho^1_\v\|^2_{L^2_x}+\|\n_{x}\p_x^\a\rho^2_\v\|^2_{L^2_x}\right)
+C\|\P^\perp \bg^\v\|_{H^s_x L^2_v(\nu)}^2+C\E_s^{\frac{1}{2}}(t)\mathbb{D}(t).
\end{aligned}
\end{equation}

Hence, by substituting \eqref{Estimate of B81}, \eqref{Estimate of B82}, \eqref{Estimate of B83-0}, \eqref{Estimate of B831} and \eqref{Estimate of B84} into \eqref{Estimate of rho-1}, we obtain
\begin{equation}\label{Estimate of rho}
\begin{aligned}
&\frac{3}{4}\left[ \|\n_x\p_x^\a \rho^1_\v\|_{L^2_x}^2+\|\n_x\p_x^\a \rho^1_\v\|_{L^2_x}^2 \right]  + \v\frac{d}{dt}\sum_{i=1}^3\Big[ \l\p_x^\a u_{\v i},\p_{x_i}\p_x^\a\rho_\v^1+\p_{x_i}\p_x^\a\rho_\v^2\r_x\\
&\qquad \qquad \qquad \qquad \qquad \qquad \qquad +\l\p_x^\a\P^\perp \bg^\v,\p_{x_i}\p_x^\a\rho_\v^1\zeta_i^1 +\p_{x_i}\p_x^\a\rho_\v^2\zeta_i^2 \r_{x,v}\Big]\\
\leq& 5 \left( \|\p_x^\a\up{div}_x u_\v\|_{L^2_x}^2 +\|\n_x\p_x^\a\theta_\v\|_{L^2_x}^2\right) + \frac{C}{\v^2}\|\mathbf{P}^\perp \bg^\v\|_{H^s_x L^2_v(\nu)}^2 +C\|\P^\perp \bg^\v\|_{H^s_x L^2_v(\nu)}^2\\
& +C\E_s^{\frac{1}{2}}(t)\mathbb{D}_s(t).
\end{aligned}
\end{equation}

\textbf{Step 4: Combine all estimates above.}

By taking $72\times \eqref{Estimate of u}+12\times \eqref{Estimate of theta}+\eqref{Estimate of rho}$, we have
\begin{equation}\label{Estimate of fluid part-0}
\begin{aligned}
&66\|\n_x\p_x^\a u_\v\|_{L^2_x}^2+\|\p_x^\a\up{div}_x u_\v\|_{L^2_x}^2+11\|\n_x\p_x^\a\theta_\v\|_{L^2_x}^2
+\frac{3}{4} \left[\|\n_x\p_x^\a \rho^1_\v\|_{L^2_x}^2+\|\n_x\p_x^\a \rho^1_\v\|_{L^2_x}^2 \right]\\
&+\v\frac{d}{dt} \Big[ \sum_{i,j=1}^3 72\l\p_{x_j}\p_x^\a\P^\perp \bg^\v,\p_x^\a u_{\v i}\zeta_{ij} \r_{x,v} + \sum_{i=1}^3 12\l\p_{x_i}\p_x^\a\P^\perp \bg^\v,\p_x^\a \theta_{\v}\zeta_{i} \r_{x,v}\\
&+\sum_{i=1}^3\l\p_x^\a u_{\v i},\p_{x_i}\p_x^\a\rho_\v^1+\p_{x_i}\p_x^\a\rho_\v^2\r_x+\l\p_x^\a\P^\perp \bg^\v,\p_{x_i}\p_x^\a\rho_\v^1\zeta_i^1 + \p_{x_i}\p_x^\a\rho_\v^2\zeta_i^2 \r_{x,v}\Big]\\
\leq & 84 \delta_1 \left[\|\p_x^\a\up{div}_x u_\v\|_{L^2_x}^2+\|\n_x\p_x^\a\theta_\v\|_{L^2_x}^2+\|\n_x\p_x^\a\rho^1_\v\|_{L^2_x}^2+\|\n_x\p_x^\a\rho^2_\v\|_{L^2_x}^2\right]\\
&+\frac{C}{\v^2}\|\mathbf{P}^\perp \bg^\v\|_{H^s_x L^2_v(\nu)}^2 + C\|\P^\perp \bg^\v\|_{H^s_x L^2_v(\nu)}^2+C\E_s^{\frac{1}{2}}(t)\mathbb{D}_s(t).
\end{aligned}
\end{equation}
Notice that
\begin{equation*}
\|\n_x\P g^\v\|_{H^{s-1}_x L^2_v}\sim \|\n_x u_\v\|_{H^{s-1}_x}^2+\|\up{div}_x u_\v\|_{H^{s-1}_x}^2+\|\n_x\theta_\v\|_{H^{s-1}_x}^2+\|\n_x\rho^1_\v\|_{H^{s-1}_x}^2+\|\n_x\rho^2_\v\|_{H^{s-1}_x}^2.
\end{equation*}
if we choose $\delta_1$ small enough such that $0<84\delta_1<\frac{1}{4}$ and sum up for $0\leq|\a|\leq s-1$, then there exist $c,\,C_2 > 0$ such that
\begin{equation}\label{Estimate of fluid part}
\begin{aligned}
\|\n_x\mathbb{P}\bg^\v\|_{H^{s-1}_xL^2_v}^2 + c\v\frac{d}{dt}E^{int}_s(g^\v)(t)\leq C_2 \left( \frac{1}{\v^2}\|\mathbf{P}^\perp \bg^\v\|_{H^s_x L^2_v(\nu)}^2+\|\mathbb{P}\bg^\v\|_{H^s_x L^2_v(\nu)}^2+\E_s^{\frac{1}{2}}(t)\mathbb{D}_s(t)\right),
\end{aligned}
\end{equation}
where $E^{int}_s(t)$ is given in \eqref{Temporary energy}.
\end{proof}


\subsubsection{Total energy estimate}
\label{subsubsec:total}

Now, we can prove the total energy estimate by combining the microscopic part and macroscopic part, i.e., Lemma \ref{Pure spatial derivative estimate} and Lemma \ref{fluid-part-energy}.

\begin{proposition}\label{A priori estimate}
Under the assumptions of Theorem \ref{Global-in-time solution of BE}, for any $0 < \v\leq 1$ and $0 \leq t \leq T$, there exist constants $C_0,\,\tilde{C}_0 > 0$ independent of $\v$ and $T$ such that
\begin{equation}\label{A priori esti}
\frac{1}{2}\frac{d}{dt}\mathcal{E}_s(t) + C_0\mathbb{D}_s(t) \leq \tilde{C}_0\mathcal{E}^{\frac{1}{2}}_s(t)\mathbb{D}_s(t),
\end{equation}
where the temporal energy functional $\mathcal{E}_s(t)$ is given in \eqref{The temporal energy functional} and dissipation functional $\mathbb{D}_s(t)$ is defined in \eqref{The energy dissipative functional}.
\end{proposition}

\begin{proof}
Choosing a larger enough $C_3$ such that $\frac{C_3\delta}{2}\geq C_2$ with $\delta,\,C_2$ being given in \eqref{Pure spatial esti} and \eqref{The dissipation of the fluid part},
and applying $C_3\times \eqref{Pure spatial esti}+\eqref{The dissipation of the fluid part}$, we find that for any $0<\v\leq 1$,
\begin{equation}\label{A priori estimate-1}
\begin{aligned}
\frac{1}{2}\frac{d}{dt}\| \bg^\v\|_{H^s_x L^2_v}^2+\frac{\delta}{2\v^2}\|\mathbf{P}^\perp \bg^\v\|_{H^s_x L^2_v(\nu)}^2& + \frac{\delta}{2}\|\P^\perp \bg^\v\|_{H^s_x L^2_v(\nu)}^2+\frac{1}{C_3}\|\n_x\P \bg^\v\|_{H^{s-1}_x L^2_v}^2\\
&+\frac{c\v}{C_3}\frac{d}{dt}E^{int}_s(t) \lesssim \E_s^{\frac{1}{2}}(t)\mathbb{D}(t)
\end{aligned}
\end{equation}
where $E^{int}_s(t)$ is defined in \eqref{Temporary energy}.

By further defining
\begin{equation}\label{The temporal energy functional}
\E_s(t):=\| \bg^\v\|_{H^s_x L^2_v}^2+\frac{2c\v}{C_3}E^{int}_s(t),
\end{equation}
and noticing that
\begin{equation}
\mathbb{D}_s(t)\sim \frac{\delta}{2\v^2}\|\mathbf{P}^\perp \bg^\v\|_{H^s_x L^2_v(\nu)}^2+\frac{\delta}{2}\|\P^\perp \bg^\v\|_{H^s_x L^2_v(\nu)}^2+\frac{1}{C_3}\|\n_x\P \bg^\v\|_{H^{s-1}_x L^2_v}^2,
\end{equation}
then we can find constants $C_0,\,\tilde{C}_0 > 0$ independent of $\v$ such that
\begin{equation}\label{A priori estimate-2}
\frac{1}{2}\frac{d}{dt}\E_s(t) + C_0 \mathbb{D}_s(t) \leq \tilde{C}_0 \E_s^{\frac{1}{2}}(t) \mathbb{D}(t).
\end{equation}
\end{proof}


\subsection{Proof of the global well-posedness}
\label{subsec:proof-global}

The global well-posedness of $\bg^\v$ to \eqref{Scaled gas mixture BE system-1}-\eqref{Initial data} in Theorem \ref{Global-in-time solution of BE} can be obtained immediately by the local well-posedness, i.e., Proposition \ref{Lmm-Local}, plus the standard continuity argument with the help of the uniform energy estimate, i.e., Proposition \ref{A priori estimate} (see \cite{JL22} for more details).

Here, we only illustrate the energy functional $\mathcal{E}_s(t)$ is continuous in $[0, T^*]$, where $T^*$ is given in Proposition \ref{Lmm-Local}.
First, there exists a $0< \v_0 \leq 1$ such that for any $0<\v\leq\v_0$,
\begin{equation}\label{Equivalence of energy functional}
\frac{1}{C_4}\mathbb{E}_s(t) \leq \mathcal{E}_s(t) \leq C_4\mathbb{E}_s(t)
\end{equation}
holds for any $t\in[0,T^*]$, where the constant $C_4 > 0$ independent of $\v$ and $T^*$.

Furthermore, by considering \eqref{A priori esti} in Proposition \ref{A priori estimate} and \eqref{The energy bound} in Proposition \ref{Lmm-Local}, we find, for any $[t_1,t_2]\subset [0,T^*]$ and $0<\v\leq \v_0\leq 1$,
\begin{equation*}\label{The continuous of energy functional}
\begin{aligned}
\Big|\mathcal{E}_s(t_2)-\mathcal{E}_s(t_1)\Big| \leq & 2\tilde{C}_0 \int_{t_1}^{t_2}\E_s^{\frac{1}{2}}(t)\mathbb{D}_s(t) \,dt \\
\leq&  2\tilde{C}_1\sup_{0\leq t\leq T^*}\E_s^{\frac{1}{2}}(t)\int_{t_1}^{t_2}\mathbb{D}_s(t) \,dt\\
\leq& 2\tilde{C}_0 \sup_{0\leq t\leq T^*}\E_s^{\frac{1}{2}}(t) \int_{t_1}^{t_2} \left(\frac{1}{\v^2}\|\P^\perp \bg^\v\|_{H^s_x L^2_v(\nu)}^2+\|\mathbf{P}^\perp \bg^\v\|_{H^s_x L^2_v(\nu)}^2 \right) \,dt\\
\leq &2\tilde{C}_0 C\sqrt{l_0 C_4} \int_{t_1}^{t_2} \left( \frac{1}{\v^2}\|\P^\perp \bg^\v\|_{H^s_x L^2_v(\nu)}^2+\|\mathbf{P}^\perp \bg^\v\|_{H^s_x L^2_v(\nu)}^2 \right) \,dt \to 0,\,\,\up{as}\,\, t_2\to t_1,
\end{aligned}
\end{equation*}
which implies the continuity of $\E_s(t)$ in $t\in[0,T^*]$.


\section{Rigorous justification of the hydrodynamic limit (Theorem \ref{Limit of Fluid equations})}
\label{sec:limit}

In this section, following \cite{JNXCJZHJ18}, we will provide a rigorous justification of the limiting process from the scaled BEGM equation \eqref{Scaled gas mixture BE system-0} to the two-fluid incompressible Navier-Stokes-Fourier system \eqref{The Boltzmann equation for gaseous mixtures NSF equ} as $\v \to 0$, i.e., the proof of Theorem \ref{Limit of Fluid equations}.

\subsection{Compactness from the uniform energy estimates}
\label{subsec:compactness}

By the uniform energy estimate \eqref{Uniform energy estimate} in Theorem \ref{Global-in-time solution of BE}, there exists a constant $C > 0$, independent of $\v$, such that for any $0<\v\leq \v_0$ and $s\geq 3$,
\begin{equation}\label{Bound of g}
\begin{aligned}
\sup_{t\geq 0}\|\bg^\v\|_{H^s_x L^2_v}^2\leq C,
\end{aligned}
\end{equation}
and
\begin{equation}\label{The energy dissipation bound of g}
\begin{aligned}
\int_0^T\|\mathbf{P}^\perp \bg^\v\|_{H^s_x L^2_v(\nu)}^2 \,dt \leq C\v^2,
\end{aligned}
\end{equation}
for any given $T>0$.

From \eqref{Bound of g}, we can find that there exists $\bg_0 \in L^\infty([0,+\infty);H^s_x L^2_v)$ such that
\begin{equation}\label{Convergence of g}
\bg^\v\to \bg_0,\quad \up{as} \quad \v\to 0
\end{equation}
weak-$\star$ for $t\geq 0$, strong in $H^{s-\eta}_x$ for any $\eta>0$, and weak in $L^2_v$. 
From \eqref{The energy dissipation bound of g}, we have
\begin{equation}\label{Convergence of of I-P g}
\mathbf{P}^\perp \bg^\v\to 0, \quad \up{in}\quad L^2([0,+\infty);H^s_x L^2_v), \quad \up{as} \quad \v\to 0.
\end{equation}

Combining the convergence of \eqref{Convergence of g} and \eqref{Convergence of of I-P g}, it holds that
\begin{equation}\label{I-P of g0}
\mathbf{P}^\perp \bg_0=0,
\end{equation}
which implies that existence of $(\brho,\bu,\btheta) \in L^\infty([0,\infty);H^s_x)$ such that
\begin{equation}\label{The form of g0}
\bg_0(t,x,v)=\brho(t,x)\M+\bu(t,x)\c v\M+\btheta(t,x) \left(\frac{|v|^2}{2}-\frac{3}{2}\right) \M.
\end{equation}


\subsection{Justification of the limiting process}
\label{subsec:justification}

Recall that
\begin{equation*}\label{The fluid variables}
\brho_{\v}=\l \bg^\v,\M\r_v, \quad \bu_{\v i}=\l \bg^\v,v_i\M\r_v, \quad \btheta_{\v}=\l \bg^\v,\left(\frac{|v|^2}{3}-1\right)\M\r_v.
\end{equation*}
By applying the convergence of $\bg^\v \to \bg_{0}$ in \eqref{Convergence of g}, we have
\begin{equation}\label{The convergence of rho-u-theta}
(\brho_\v ,\bu_\v ,\btheta_\v )\to (\brho,\bu,\btheta),\quad \up{as} \quad \v\to 0,
\end{equation}
weakly-$\star$ for $t\geq 0$, strongly in $H^{s-\eta}_x$ for any $\eta>0$.

Next, multiplying \eqref{Scaled gas mixture BE system-1} by $\psi_1$, $\psi_2$, $\psi_{i+5}$, $\psi_9$ and $\psi_{10}$ in \eqref{Space of Ker}, and integrating over $v \in \mathbb{R}^3$, it leads to
\begin{equation}\label{The local conservation laws}
\left\{
\begin{aligned}
&\brho_\v+\frac{1}{\v}\up{div}_x \bu_\v=0,\\
&\p_t \bu_\v+\frac{1}{\v}\n_x(\brho_\v +\btheta_\v)+\frac{1}{\v}\up{div}_x\l\hat{A}\M,L \bg^\v\r_v+\l\mathcal{L}(\bg^\v,\bg^\v),v\M\r_v\\
&\qquad\qquad\qquad\qquad\qquad\qquad\qquad\qquad\qquad\qquad =\v\l\tilde{\Gamma}(\bg^\v,\bg^\v),v\M\r_v,\\
&\p_t \btheta_\v+\frac{2}{3\v}\up{div}\bu_\v+\frac{2}{3\v}\up{div}_x\l\hat{B}\M,L\bg^\v\r_v+\l\mathcal{L}(\bg^\v,\bg^\v),(\frac{|v|^2}{3}-1)\M\r_v\\
&\qquad\qquad\qquad\qquad\qquad\qquad\qquad\qquad\qquad\qquad=\v\l\tilde{\Gamma}(\bg^\v,\bg^\v),\left(\frac{|v|^2}{3}-1\right)\M\r_v.\\
\end{aligned}
\right.
\end{equation}

\textbf{The incompressibility and Boussinesq relation.}

From the energy bound \eqref{Bound of g}, we deduce that
\begin{equation}\label{Convergence of partial rho}
\v\p_t\brho_\v\to 0, \quad \up{as} \quad \v\to 0,
\end{equation}
in the sense of distribution.

Combining with $\eqref{The local conservation laws}_1$, $\eqref{The local conservation laws}_2$ and \eqref{Convergence of partial rho}, we have
\begin{equation}\label{The incompressibility-0}
\up{div}_x \bu_\v\to 0, \quad \up{as} \quad \v \to 0,
\end{equation}
in the sense of distribution, which further yields
\begin{equation}\label{The incompressibility}
\up{div}_x \bu= 0,
\end{equation}
in the sense of distribution.

By applying the equation $\eqref{The local conservation laws}_3$, $\eqref{The local conservation laws}_4$, the energy bound \eqref{Bound of g}, and the energy dissipation bound \eqref{The energy dissipation bound of g}, it shows that
\begin{equation*}
\begin{aligned}
&\n_x(\brho_\v+\btheta_\v)\\
=&-\v\p_t \bg^\v -\up{div}_x\l\hat{A}\M,L(\mathbf{P}^\perp \bg^\v)\r_v-\v\l\mathcal{L}(\bg^\v,\bg^\v),v\M\r_v +\v^2\l\tilde{\Gamma}(\bg^\v,\bg^\v),v\M\r_v\\
\to & \  0, \quad \up{as} \quad \v\to 0,
\end{aligned}
\end{equation*}
in the sense of distribution, which implies that
\begin{equation}\label{Boussinesq relation}
\begin{aligned}
\n_x(\brho+\btheta)=0,
\end{aligned}
\end{equation}
in the sense of distribution.

\textbf{The convergence of $\frac{3}{5}\btheta_\v-\frac{2}{5}\brho_\v$.}

By applying $\frac{3}{5}\eqref{The local conservation laws}_5-\frac{2}{5}\eqref{The local conservation laws}_1$ and $\frac{3}{5}\eqref{The local conservation laws}_6-\frac{2}{5}\eqref{The local conservation laws}_2$, we have
\begin{equation}\label{Equation of theta and rho}
\begin{aligned}
\p_t \left(\frac{3}{5}\btheta_\v-\frac{2}{5}\brho_\v\right) = &-\frac{2}{5\v}\up{div}\l\hat{B}\M,L\bg^\v\r_v-\frac{3}{5}\l\mathcal{L}(\bg^\v,\bg^\v),\, \left(\frac{|v|^2}{3}-1\right)\M\r_v\\
&+\v\l\tilde{\Gamma}(\bg^\v,\bg^\v),\, \left(\frac{|v|^2}{3}-1\right)\M\r_v.
\end{aligned}
\end{equation}

For $t\in[0,+\infty)$ a.e., it follows the energy bound \eqref{Bound of g} that
\begin{equation}\label{Bound of theta and rho}
\|\frac{3}{5}\btheta_\v - \frac{2}{5}\brho_\v\|^2_{H^s_x} \lesssim 1,
\end{equation}
which further implies there exist $\tilde{\btheta} \in L^\infty([0,+\infty);H^s_x)$ such that
\begin{equation*}
\frac{3}{5}\btheta_\v-\frac{2}{5}\brho_\v\to\tilde{\btheta}
\end{equation*}
weakly-$\star$ for $t\geq 0$, strongly in $H^{s-\eta}_x$ for any $\eta>0$.

For any $[t_1,t_2]\in[0,+\infty)$, $0\leq|\a|\leq s-1$ and test function $\xi(x)\in C^\infty_0(\R^3)$, it follows the uniform energy estimate \eqref{Bound of g} that
\begin{equation}\label{The equi-continuity of theta-rho}
\begin{aligned}
&\left|\int_{\R^3} \left[\p_x^\a \left(\frac{3}{5}\btheta_\v-\frac{2}{5}\brho_\v \right)(t_2,x) - \p_x^\a \left(\frac{3}{5}\btheta_\v - \frac{2}{5}\brho_\v \right)(t_1,x ) \right]\xi(x) \,dx \right|\\[4pt]
\leq& \left|\frac{2}{5\v}\int_{t_1}^{t_2}\int_{\R^3}\l\hat{B}\M,L(\p_x^{\a}\n_x \mathbf{P}^\perp \bg^\v)\r_v \, \xi \,dx\,dt\right|\\[4pt]
&\quad+\left|\frac{3}{5}\int_{t_1}^{t_2}\int_{\R^3}\l L(\p_x^{\a}\n_x \mathbf{P}^\perp \bg^\v,\, \p_x^{\a}\n_x \mathbf{P}^\perp \bg^\v),\, \left(\frac{|v|^2}{3}-1\right)\M\r_v \, \xi \,dx\,dt \right|\\[4pt]
&\quad + \left|\v\int_{t_1}^{t_2}\int_{\R^3}\l\n_x^\a\tilde{\Gamma}(\bg^\v,\bg^\v),\left(\frac{|v|^2}{3}-1\right) \M\r_v \, \xi \,dx\,dt \right|\\[4pt]
\lesssim & \frac{1}{\v^2} \int_{t_1}^{t_2} \left( \|\mathbf{P}^\perp \bg^\v\|_{H^s_x L^2_v(\nu)}^2 + \|\mathbf{P}^\perp \bg^\v\|_{H^s_x L^2_v(\nu)}^2\right) \,dt,
\end{aligned}
\end{equation}
where we substitute \eqref{Equation of theta and rho} in the first inequality above.

Therefore, the energy dissipation bound \eqref{The energy dissipation bound of g} plus \eqref{The equi-continuity of theta-rho} implies the equip-continuity of \eqref{Equation of theta and rho} in $t$. Then, combining with \eqref{Bound of theta and rho} and from Arzel$\grave{\text{a}}$-Ascoli theorem, we obtain, for any $\eta>0$,
\begin{equation*}\label{The space of theta}
\tilde{\btheta} \in C([0,+\infty);H^{s-1-\eta})\cap L^\infty([0,+\infty);H^{s-\eta}),
\end{equation*}
and
\begin{equation}\label{Convergence of theta and rho}
\begin{aligned}
\frac{3}{5}\btheta_\v-\frac{2}{5}\brho_\v\to \tilde{\btheta} \quad \up{in} \quad C([0,+\infty);H^{s-1-\eta})\cap L^\infty([0,+\infty);H^{s-\eta}),
\end{aligned}
\end{equation}
as $\v\to 0$.

Thus, by using $(\btheta_\v,\brho_\v) \to (\btheta,\brho)$ in \eqref{The convergence of rho-u-theta} and \eqref{Convergence of theta and rho}, we have
\[
\tilde{\btheta}=\frac{3}{5}\btheta-\frac{2}{5}\brho.
\]
Since $\btheta=(\frac{3}{5}\btheta-\frac{2}{5}\brho)$ and \eqref{Boussinesq relation}, it is direct to show that $\tilde{\btheta}=\btheta$ and $\brho+\btheta=\mathbf{0}$ in whole space.

\textbf{Convergence of $\mathcal{P}\bu_\v$.}

Applying the Leray projection operator $\mathcal{P}$ to $\eqref{The local conservation laws}_3$ and $\eqref{The local conservation laws}_4$, we have
\begin{equation*}\label{Equation of ul}
\begin{aligned}
\p_t\mathcal{P}\bu_\v+\frac{1}{\v}\mathcal{P}\up{div}_x\l\hat{A}\M,L \bg^\v\r_v+\mathcal{P}\l\mathcal{L}(\bg^\v,\bg^\v),v\M\r_v=\v\mathcal{P}\l\tilde{\Gamma}(\bg^\v,\bg^\v),v\M\r_v
\end{aligned}
\end{equation*}
By using the similar argument as in showing the convergence in \eqref{Convergence of theta and rho}, we can find the divergence-free $\tilde{\bu}\in L^\infty([0,+\infty);H^s_x)$ such that, for any $\eta>0$,
\begin{equation}\label{Convergence of u}
\begin{aligned}
\mathcal{P}\bu_\v \to \tilde{\bu}\qquad \up{in} \quad C([0,+\infty);H^{s-1-\eta}_x)\cap L^\infty([0,+\infty);H^{s-\eta}_x),
\end{aligned}
\end{equation}
as $\v\to 0$. \\
Furthermore, \eqref{The convergence of rho-u-theta} and \eqref{Convergence of u} leads to $\tilde{\bu}=\mathcal{P}\bu$, whereas $\tilde{\bu}=\bu$ can be shown by considering \eqref{The incompressibility}.

According to \cite{BCGFLD91,Saint-Raymond19}, the system \eqref{The local conservation laws} can be rewritten as, for $l,n \in \{1,2\}$ and $l\neq n$,
\begin{equation}\label{The local conservation laws-2}
\left\{
\begin{aligned}
&\brho_\v+\frac{1}{\v}\up{div}_x \bu_\v=0,\\[4pt]
&\p_t u_{\v l}+\frac{1}{\v}\n_x(\rho_{\v l}+\theta_{\v l}) + \up{div}_x \left(u_{\v l}\otimes u_{\v l}-\frac{|u_{\v l}|^2}{3}I\right) + \frac{1}{\sigma}(u_{\v l}-u_{\v n})\\[4pt]
&\qquad\qquad\qquad\qquad\qquad\qquad\qquad\qquad\quad=\mu\up{div}_x\Sigma (u_{\v l})+\up{div}_x R_{u_\v},\\[4pt]
&\p_t \theta_{\v l}+\frac{2}{3\v}\up{div}u_{\v l}+\up{div}_x( u_{\v l}\theta_{\v l})+\frac{1}{\lambda}\left(\theta_{\v l}-\theta_{\v n}\right)\\[4pt]
&\qquad\qquad\qquad\qquad\qquad\qquad\qquad\qquad\quad=\kappa\Delta_x\theta_{\v l}+\up{div}_x R_{\theta_\v},
\end{aligned}
\right.
\end{equation}
where the constants $\mu,\,\kappa,\,\sigma,\,\lambda$ are given by \eqref{The constants of nu and kappa} and \eqref{The constants of sigma and lambda}, and $\Sigma(u_{\v l})=\n_x u_{\v l}+(\n_x u_{\v l})^\top-\frac{2}{3}\up{div}_x u_{\v l}I$, and $R_{u_\v},\,R_{\theta_\v}$ have the following form
\begin{equation}\label{The special form of R}
\begin{aligned}
R_{u_\v},\,R_{\theta_\v} = &-\v \l\p_t g^\v_l,\,\varphi \r_v - \l v\c\n_xP^\perp g^\v_l,\,\varphi \r_v + \v\l\hat{L}g^\v_l,\,\varphi \r_v -\v\l\hat{L}(g^\v_l,g^\v_n),\, \varphi \r_v\\[4pt]
&+\l\hat{\Gamma}(P^\perp g^\v_l,P^\perp g^\v_l), \, \varphi \r_v +\l\hat{\Gamma}(P^\perp g^\v_l,P g^\v_l),\,\varphi \r_v
+\l\hat{\Gamma}(P g^\v_l,P^\perp g^\v_l),\,\varphi \r_v\\[4pt]
&+\v^2\l\hat{\Gamma}(g^\v_l,g^\v_n), \, \varphi \r_v + \v\l\hat{\Gamma}(g^\v_l,g^\v_n), \, v \M\r_v,
\end{aligned}
\end{equation}
with $\varphi =\hat{A}\M$ for $R_{u_\v}$ and $\varphi =\hat{B}\M$ for $R_{\theta_\v}$.

\textbf{The equations of $\btheta$ and $\bu$.}

Decomposing $\bu_\v$ into $\bu_\v=\mathcal{P}\bu_\v+\mathcal{Q}\bu_\v$ with $\mathcal{Q}=\n_x\Delta_x^{-1}\up{div}_x$, 
and taking the Leray projection $\mathcal{P}$ to $\eqref{The local conservation laws-2}_2$ yield that
\begin{equation*}
\begin{aligned}
\p_t\mathcal{P}u_{\v l}+\mathcal{P}\up{div}_x(\mathcal{P}u_{\v l}\otimes\mathcal{P}u_{\v l})+\frac{1}{\sigma}(\mathcal{P}u_{\v l}-\mathcal{P}u_{\v n})-\mu\Delta_x\mathcal{P}u_{\v l}=\mathcal{P}\up{div}_x\tilde{R}_{ u_\v},
\end{aligned}
\end{equation*}
where
\begin{equation}\label{The definition of tilde-R-u}
\tilde{R}_{u_\v}=R_{ u_\v}-\mathcal{P}\up{div}_x(\mathcal{P}u_{\v l}\otimes\mathcal{Q}u_{\v l}+\mathcal{Q}u_{\v l}\otimes\mathcal{P}u_{\v l}+\mathcal{Q}u_{\v l}\otimes\mathcal{Q}u_{\v l}).
\end{equation}

For any $T>0$, let the vector-valued test function $\bm{\xi}(t,x)=(\xi_1(t,x),\xi_{2}(t,x),\xi_3(t,x))$ satisfying
\begin{equation*}
\xi_i(t,x)\in C^1([0,T],C_0^\infty(\R^3)) \quad \up{with} \quad \xi_i(0,x)=1
\end{equation*}
and
\begin{equation*}
\xi_i(t,x)=0 \quad\up{for}\quad t\geq T',\quad \up{with}\quad T'<T 
\end{equation*}
for $i=1,2,3$ and $\up{div}_x \bm{\xi}(t,x)=0$ for any $t$.

Combining the uniform energy estimates \eqref{Bound of g}-\eqref{The energy dissipation bound of g}, the definition of $R_{u_\v}$ in \eqref{The special form of R} and $\tilde{R}_{u\v}$ in \eqref{The definition of tilde-R-u}, we have
\begin{equation*}
\int_0^T
(\mathcal{P}\up{div}_x\tilde{R}_{u_\v}) \c \bm{\xi} \,dx \,dt \to 0 \quad \up{as}\quad \v\to 0,
\end{equation*}
and, for $l,n \in \{1,2\}$ and $l\neq n$,
\begin{equation*}
\begin{aligned}
&\int_0^T \int_{\R^3} \Big[\p_t\mathcal{P}u_{\v l}+\mathcal{P}\up{div}_x(\mathcal{P}u_{\v l}\otimes\mathcal{P}u_{\v l})+\frac{1}{\sigma}(\mathcal{P}u_{\v l}-\mathcal{P}u_{\v n})-\mu\Delta_x\mathcal{P}u_{\v l}\Big] \c \bm{\xi} \,dx \,dt \\[4pt]
\to & -\int_{\R^3} u_{0l} \c \bm{\xi}(0,x) \, dx - \int_0^T \int_{\R^3}u_l \c\p_t\bm{\xi} +u_l \otimes u_l:\n_x\bm{\xi} - \frac{1}{\sigma}(u_l-u_n)\c\bm{\xi} - \mu u_l\c\Delta_x \bm{\xi} \,dx \,dt,
\end{aligned}
\end{equation*}
as $\v\to 0$.

Notice that $\tilde{\theta}_{\v l}=\frac{3}{5}\theta_{\v l}-\frac{2}{5}\rho_{\v l}$, then it follows from $\eqref{The local conservation laws-2}_1$ and $\eqref{The local conservation laws-2}_3$ that
\begin{equation*}
\p_t\tilde{\theta}_{\v l}+\frac{3}{5}\up{div}(\mathcal{P}u_{\v l}\theta_{\v l})+\frac{3}{5\lambda}(\theta_{\v l}-\theta_{\v n})-\frac{3}{5}\kappa\Delta_x\tilde{\theta}_{\v l}=\up{div}_x\tilde{R}_{\theta_\v},
\end{equation*}
where
\begin{equation*}
\tilde{R}_{\theta_\v}=\frac{3}{5}R_{\theta_\v}-\frac{3}{5}\up{div}_x(\mathcal{Q}u_{\v l}\theta_{\v l}).
\end{equation*}

For any $T>0$, let the test function $\xi(t,x)$ satisfy
\begin{equation*}
\xi(t,x)\in C^1([0,T],C_0^\infty(\R^3)) \quad \up{with} \quad \xi(0,x)=1 
\end{equation*}
and
\begin{equation*}
\xi(t,x)=0 \quad \up{for}\quad t \geq T',\quad \up{for} \quad T'<T.
\end{equation*}

Considering the uniform estimates \eqref{Bound of g}-\eqref{The energy dissipation bound of g} as well as the convergence \eqref{Convergence of theta and rho}-\eqref{Convergence of u}, we have
\begin{equation*}
\int_0^T \up{div}_x\tilde{R}_{\theta_\v}\xi(t,x) \,dx \,dt \to 0, \quad \up{as} \quad \v\to 0,
\end{equation*}
and, for $l,n \in \{1,2\}$ and $l\neq n$,
\begin{equation*}
\begin{aligned}
&\int_0^T \int_{\R^3} \left[\p_t\tilde{\theta}_{\v l}+\frac{3}{5}\up{div}(\mathcal{P}u_{\v l}\theta_{\v l})+\frac{3}{5\lambda}(\theta_{\v l}-\theta_{\v n})-\frac{3}{5}\kappa\Delta_x\tilde{\theta}_{\v l}\right] \xi \,dx\,dt \\[4pt]
\to & -\int_{\R^3} \left(\frac{3}{5}\theta_{0l}-\frac{2}{5}\rho_{0l}\right) \xi(0,x)\, dx - \int_0^T \int_{\R^3}u_l\c\p_t\xi + u_l\theta_l\c\n_x\xi+
\frac{1}{\sigma}(\theta_l-\theta_n)\xi-\kappa \theta_l\c\Delta_x\xi \,dx\,dt,
\end{aligned}
\end{equation*}
as $\v\to 0$.

Finally, combining all the convergence results above, we obtain that
\begin{equation*}
(\theta_l,u_l)\in C([0,+\infty);H^{s-1}_x)\cap L^\infty([0,+\infty);H^s_x)
\end{equation*}
which satisfy the two-fluids incompressible Navier-Stokes equations
\begin{equation*}\label{The equation of theta and u}
\left\{
\begin{aligned}
&\p_t u_{\v l}+u_l\c\n_x u_l+\frac{1}{\sigma}(u_{\v l}-u_{\v n})+\n_x p_l=\mu\Delta_x u_l,\\
&\p_t \theta_{\v l}+u_\v\c\n_x\theta_l+\frac{1}{\lambda}(\theta_{\v l}-\theta_{\v n})=\kappa\Delta_x\theta_l,\\
&\up{div}_x u_l=0,
\end{aligned}
\right.
\end{equation*}
with initial data
\begin{equation*}
u_l(0,x)=\mathcal{P}u_{0l}(x), \quad \theta_l(0,x)=\frac{3}{5}\theta_{0l}(x)-\frac{2}{5}\rho_{0l}(x)
\end{equation*}
for $l,n \in \{1,2\}$ and $l\neq n$.

\section*{Acknowledgment}

K.~Qi is supported by grants from the School of Mathematics at the University of Minnesota.

\bibliographystyle{siam}
\bibliography{Kinetic-Fluid_Project}

\begin{thebibliography}{10}

\bibitem{ARVC02}
{\sc R.~Alexandre and C.~Villani}, {\em On the {B}oltzmann equation for long-range interactions}, Communications on Pure and Applied Mathematics, 55 (2002), pp.~30--70.

\bibitem{ABT03}
{\sc K.~Aoki, C.~Bardos, and S.~Takata}, {\em Knudsen layer for gas mixtures}, Journal of Statistical Physics, 112 (2003), pp.~629--655.

\bibitem{AD12}
{\sc D.~Ars\'{e}nio}, {\em From {B}oltzmann's equation to the incompressible {N}avier-{S}tokes-{F}ourier system with long-range interactions}, Archive for Rational Mechanics and Analysis, 206 (2012), pp.~367--488.

\bibitem{Saint-Raymond19}
{\sc D.~Ars\'{e}nio and L.~Saint-Raymond}, {\em From the {V}lasov-{M}axwell-{B}oltzmann system to incompressible viscous electro-magneto-hydrodynamics. {V}ol. 1}, EMS Monographs in Mathematics, European Mathematical Society (EMS), Z\"{u}rich, 2019, pp.~xii+406.

\bibitem{BCGFLCD93}
{\sc C.~Bardos, F.~Golse, and C.~D. Levermore}, {\em Fluid dynamic limits of kinetic equations. {II}. {C}onvergence proofs for the {B}oltzmann equation}, Comm. Pure Appl. Math., 46 (1993), pp.~667--753.

\bibitem{BCGFLD91}
{\sc C.~Bardos, F.~Golse, and D.~Levermore}, {\em Fluid dynamic limits of kinetic equations. {I}. {F}ormal derivations}, Journal of Statistical Physics, 63 (1991), pp.~323--344.

\bibitem{BCUS91}
{\sc C.~Bardos and S.~Ukai}, {\em The classical incompressible {N}avier-{S}tokes limit of the {B}oltzmann equation}, Mathematical Models and Methods in Applied Sciences, 1 (1991), pp.~235--257.

\bibitem{BY12}
{\sc C.~Bardos and X.~Yang}, {\em The classification of well-posed kinetic boundary layer for hard sphere gas mixtures}, Communications in Partial Differential Equations, 37 (2012), pp.~1286--1314.

\bibitem{BD15}
{\sc C.~Bianca and C.~Dogbe}, {\em On the {B}oltzmann gas mixture equation: linking the kinetic and fluid regimes}, Communications in Nonlinear Science and Numerical Simulation, 29 (2015), pp.~240--256.

\bibitem{BM15}
{\sc M.~Briant}, {\em From the {B}oltzmann equation to the incompressible {N}avier–{S}tokes equations on the torus: a quantitative error estimate}, Journal of Differential Equations, 259 (2015), pp.~6072--6141.

\bibitem{BD16}
{\sc M.~Briant and E.~S. Daus}, {\em The {B}oltzmann equation for a multi-species mixture close to global equilibrium}, Archive for Rational Mechanics and Analysis, 222 (2016), pp.~1367--1443.

\bibitem{CRE80}
{\sc R.~E. Caflisch}, {\em The fluid dynamic limit of the nonlinear {B}oltzmann equation}, Communications on Pure and Applied Mathematics, 33 (1980), pp.~651--666.

\bibitem{Cercignani87}
{\sc C.~Cercignani}, {\em The {B}oltzmann equation and its applications}, vol.~67 of Applied Mathematical Sciences, Springer-Verlag, New York, 1988.

\bibitem{DMAERLJL89}
{\sc A.~De~Masi, R.~Esposito, and J.~L. Lebowitz}, {\em Incompressible {N}avier-{S}tokes and {E}uler limits of the {B}oltzmann equation}, Communications on Pure and Applied Mathematics, 42 (1989), pp.~1189--1214.

\bibitem{DLGF94}
{\sc L.~Desvillettes and F.~Glose}, {\em A remark concerning the {C}hapman-{E}nskog asymptotics}, in Advances in kinetic theory and computing, vol.~22 of Ser. Adv. Math. Appl. Sci., World Sci. Publ., River Edge, NJ, 1994, pp.~191--203.

\bibitem{DRLPL89}
{\sc R.~J. DiPerna and P.-L. Lions}, {\em On the cauchy problem for {B}oltzmann equations: Global existence and weak stability}, Annals of Mathematics, 130 (1989), p.~321–366.

\bibitem{Dogbe08}
{\sc C.~Dogbe}, {\em Fluid dynamic limits for gas mixture. {I}. {F}ormal derivations}, Mathematical Models and Methods in Applied Sciences, 18 (2008), pp.~1633--1672.

\bibitem{DRJLWXLLQ22}
{\sc R.~Duan, W.-X. Li, and L.~Liu}, {\em Gevrey regularity of mild solutions to the non-cutoff {B}oltzmann equation}, Advances in Mathematics, 395 (2022), p.~108159.

\bibitem{DRJLSQSSSRM21}
{\sc R.~Duan, S.~Liu, S.~Sakamoto, and R.~M. Strain}, {\em Global mild solutions of the {L}andau and non-cutoff {B}oltzmann equations}, Communications on Pure and Applied Mathematics, 74 (2021), pp.~932--1020.

\bibitem{GFSRL05}
{\sc G.~Francois and L.~Saint-Raymond}, {\em Hydrodynamic limits for the {B}oltzmann equation}, Rivista di Matematica della Universita di Parma, 4** (2005), pp.~1--144.

\bibitem{GT2020}
{\sc I.~Gallagher and I.~Tristani}, {\em On the convergence of smooth solutions from {B}oltzmann to {N}avier-{S}tokes}, Annales Henri Lebesgue, 3 (2020), pp.~561--614.

\bibitem{GFSRL04}
{\sc F.~Golse and L.~Saint-Raymond}, {\em The {N}avier-{S}tokes limit of the {B}oltzmann equation for bounded collision kernels}, Inventiones mathematicae, 155 (2004), pp.~81--161.

\bibitem{GFSRL09}
\leavevmode\vrule height 2pt depth -1.6pt width 23pt, {\em The incompressible {N}avier–{S}tokes limit of the {B}oltzmann equation for hard cutoff potentials}, Journal de Math\'{e}matiques Pures et Appliqu\'{e}es, 91 (2009), pp.~508--552.

\bibitem{GPTSRM11}
{\sc P.~T. Gressman and R.~M. Strain}, {\em Global classical solutions of the {B}oltzmann equation without angular cut-off}, Journal of the American Mathematical Society, 24 (2011), pp.~771--847.

\bibitem{GY03}
{\sc Y.~Guo}, {\em Classical solutions to the {B}oltzmann equation for molecules with an angular cutoff}, Archive for Rational Mechanics and Analysis, 169 (2003).

\bibitem{GuoYan03}
\leavevmode\vrule height 2pt depth -1.6pt width 23pt, {\em The vlasov-maxwell-boltzmann system near maxwellians}, Inventiones mathematicae, 153 (2003), pp.~593--630--468.

\bibitem{GY04}
\leavevmode\vrule height 2pt depth -1.6pt width 23pt, {\em The {B}oltzmann equation in the whole space}, Indiana University Mathematics Journal, 53 (2004).

\bibitem{GY06}
\leavevmode\vrule height 2pt depth -1.6pt width 23pt, {\em Boltzmann diffusive limit beyond the {N}avier-{S}tokes approximation}, Communications on Pure and Applied Mathematics, 59 (2006), pp.~626--687.

\bibitem{GY10}
\leavevmode\vrule height 2pt depth -1.6pt width 23pt, {\em Decay and continuity of the {B}oltzmann equation in bounded domains}, Archive for Rational Mechanics and Analysis, 197 (2010), pp.~713--809.

\bibitem{GYHFMWY21}
{\sc Y.~Guo, F.~Huang, and Y.~Wang}, {\em Hilbert expansion of the {B}oltzmann equation with specular boundary condition in half-space}, Archive for Rational Mechanics and Analysis, 241 (2021), pp.~231--309.

\bibitem{GYJJJN09}
{\sc Y.~Guo, J.~Jang, and N.~Jiang}, {\em Local {H}ilbert expansion for the {B}oltzmann equation}, Kinetic and Related Models, 2 (2009), pp.~205--214.

\bibitem{GYJJJN10}
\leavevmode\vrule height 2pt depth -1.6pt width 23pt, {\em Acoustic limit for the {B}oltzmann equation in optimal scaling}, Communications on Pure and Applied Mathematics, 63 (2010), pp.~337--361.

\bibitem{Hilbert1912}
{\sc D.~Hilbert}, {\em Begründung der kinetischen gastheorie (german)}, Mathematische Annalen, 72 (1912), pp.~562--577.

\bibitem{ICSL20}
{\sc C.~Imbert and L.~Silvestre}, {\em The weak {H}arnack inequality for the {B}oltzmann equation without cut-off}, Journal of the European Mathematical Society (JEMS), 22 (2020), pp.~507--592.

\bibitem{ICSLE22}
{\sc C.~Imbert and L.~E. Silvestre}, {\em Global regularity estimates for the {B}oltzmann equation without cut-off}, Journal of the American Mathematical Society, 35 (2022), pp.~625--703.

\bibitem{JJJN09}
{\sc J.~Jang and N.~Jiang}, {\em Acoustic limit of the {B}oltzmann equation: classical solutions}, Discrete and Continuous Dynamical Systems. Series A, 25 (2009), pp.~869--882.

\bibitem{JJKC21}
{\sc J.~Jang and C.~Kim}, {\em Incompressible {E}uler limit from {B}oltzmann equation with diffuse boundary condition for analytic data}, Annals of PDE, 7 (2021).

\bibitem{JNLCDMN10}
{\sc N.~Jiang, C.~D. Levermore, and N.~Masmoudi}, {\em Remarks on the acoustic limit for the {B}oltzmann equation}, Communications in Partial Differential Equations, 35 (2010), pp.~1590--1609.

\bibitem{JL22}
{\sc N.~Jiang and Y.-L. Luo}, {\em From {V}lasov-{M}axwell-{B}oltzmann system to two-fluid incompressible {N}avier-{S}tokes-{F}ourier-{M}axwell system with {O}hm's law: convergence for classical solutions}, Annals of PDE. Journal Dedicated to the Analysis of Problems from Physical Sciences, 8 (2022).

\bibitem{JNMN17}
{\sc N.~Jiang and N.~Masmoudi}, {\em Boundary layers and incompressible {N}avier-{S}tokes-{F}ourier limit of the {B}oltzmann equation in bounded domain {I}}, Communications on Pure and Applied Mathematics, 70 (2017), pp.~90--171.

\bibitem{JNXLJ15}
{\sc N.~Jiang and L.~Xiong}, {\em Diffusive limit of the {B}oltzmann equation with fluid initial layer in the periodic domain}, SIAM Journal on Mathematical Analysis, 47 (2015), pp.~1747--1777.

\bibitem{JNXCJZHJ18}
{\sc N.~Jiang, C.-J. Xu, and H.~Zhao}, {\em Incompressible {N}avier-{S}tokes-{F}ourier limit from the {B}oltzmann equation: classical solutions}, Indiana University Mathematics Journal, 67 (2018), pp.~1817--1855.

\bibitem{KCLD18}
{\sc C.~Kim and D.~Lee}, {\em The {B}oltzmann equation with specular boundary condition in convex domains}, Communications on Pure and Applied Mathematics, 71 (2018), pp.~411--504.

\bibitem{HLBJJC23}
{\sc H.~Lingbing and J.~Jin-Cheng}, {\em On the {C}auchy problem for the cutoff {B}oltzmann equation with small initial data}, Journal of Statistical Physics, 190 (2023), p.~52.

\bibitem{LYY04}
{\sc T.-P. Liu, T.~Yang, and S.-H. Yu}, {\em Energy method for {B}oltzmann equation}, Physica D. Nonlinear Phenomena, 188 (2004), pp.~178--192.

\bibitem{MNSRL03}
{\sc N.~Masmoudi and L.~Saint-Raymond}, {\em From the {B}oltzmann equation to the {S}tokes-{F}ourier system in a bounded domain}, Communications on Pure and Applied Mathematics, 56 (2003), pp.~1263--1293.

\bibitem{MS10}
{\sc S.~Mischler}, {\em Kinetic equations with {M}axwell boundary conditions}, Annales scientifiques de l'\'{E}cole Normale Sup\'{E}rieure, 43 (2010), pp.~719--760.

\bibitem{NT78}
{\sc T.~Nishida}, {\em Fluid dynamical limit of the nonlinear {B}oltzmann equation to the level of the compressible {E}uler equation}, Communications in Mathematical Physics, 61 (1978), pp.~119--148.

\bibitem{Qi21_soft}
{\sc K.~Qi}, {\em On the measure valued solution to the inelastic {B}oltzmann equation with soft potentials}, Journal of Statistical Physics, 183 (2021).

\bibitem{Qi22}
\leavevmode\vrule height 2pt depth -1.6pt width 23pt, {\em Measure valued solution to the spatially homogeneous {B}oltzmann equation with inelastic long-range interactions}, Journal of Mathematical Physics, 63 (2022), pp.~021503, 22.

\bibitem{ARMYUSXCJYT12}
{\sc A.~Radjesvarane, M.~Yoshinori, U.~Seiji, C.~J. Xu, and Y.~Tong}, {\em The {B}oltzmann equation without angular cutoff in the whole space:{I}, global existence for soft potential}, Journal of Functional Analysis, 262 (2012), pp.~915--1010.

\bibitem{SRL09}
{\sc L.~Saint-Raymond}, {\em Hydrodynamic limits of the {B}oltzmann equation}, Springer-Verlag, Berlin, 1~ed., 2009.

\bibitem{SY10}
{\sc A.~Sotirov and S.-H. Yu}, {\em On the solution of a {B}oltzmann system for gas mixtures}, Archive for Rational Mechanics and Analysis, 195 (2010), pp.~675--700.

\bibitem{US74}
{\sc S.~Ukai}, {\em On the existence of global solutions of mixed problem for non-linear {B}oltzmann equation}, Proceedings of the Japan Academy, 50 (1974), pp.~179--184.

\bibitem{Villani02}
{\sc C.~Villani}, {\em A review of mathematical topics in collisional kinetic theory}, in Handbook of mathematical fluid dynamics, {V}ol. {I}, North-Holland, Amsterdam, 2002, pp.~71--305.

\bibitem{Wang13}
{\sc Y.~Wang}, {\em Decay of the two-species {V}lasov-{P}oisson-{B}oltzmann system}, Journal of Differential Equations, 254 (2013), pp.~2304--2340.

\bibitem{WuYang23}
{\sc T.~Wu and X.~Yang}, {\em Hydrodynamic limit of {B}oltzmann equations for gas mixture}, Journal of Differential Equations, 377 (2023), pp.~418--468.

\end{thebibliography}

\end{document}